\documentclass[10pt,reqno]{amsart}

\usepackage{amssymb}
\usepackage{amsfonts}
\usepackage{amsthm}
\usepackage{amsmath}
\usepackage{bbm}
\numberwithin{equation}{section}
\allowdisplaybreaks

\newtheorem{theorem}{Theorem}[section]
\newtheorem{proposition}[theorem]{Proposition}

\newtheorem{lemma}[theorem]{Lemma}
\newtheorem{corollary}[theorem]{Corollary}

\theoremstyle{definition}

\theoremstyle{remark}
\newtheorem{remark}[theorem]{Remark}

\newcommand{\R}{\mathbb{R}}
\newcommand{\Q}{\mathbb{Q}}

\newcommand{\N}{\mathbb{N}}
\newcommand{\Z}{\mathbb{Z}}
\newcommand{\C}{\mathbb{C}}
\newcommand{\1}{\mathbbm{1}}

\renewcommand{\hat}{\widehat}

\newcommand{\scriptA}{\mathcal{A}}

\newcommand{\Sum}{\displaystyle\sum}

\newcommand{\lsm}{\lesssim}
\newcommand{\di}{d_h}

\newcommand{\h}{h(\varphi)}

\DeclareMathOperator*{\supp}{supp}

\begin{document}

\title[On $L^p$-improving for averages associated to hypersurfaces]{On $L^p$-improving for averages associated to mixed homogeneous polynomial hypersurfaces in $\R^3$}
\author{Spyridon Dendrinos}
\address{School of Mathematical Sciences, University College Cork, Cork, Ireland} 
\email{sd@ucc.ie}
\author{Eugen Zimmermann}
\address{School of Mathematical Sciences, University College Cork, Cork, Ireland}
\email{eugen.zimmermann@ucc.ie}

\thanks{This research was supported by grant DE1517/2-1 of the Deutsche Forschungsgemeinschaft.}

\begin{abstract}
We establish $L^p-L^q$ estimates for averaging operators associated to mixed homogeneous polynomial hypersurfaces in $\R^3$. These are described in terms of the mixed homogeneity and the order of vanishing of the polynomial hypersurface and its Gaussian curvature transversally to their zero sets.
\end{abstract}

\maketitle
\section{Introduction}

This article deals with the fundamental problem of determining the precise amount of $L^p$-improving under convolution with surface measure supported on hypersurfaces in $\R^3$, given by graphs of mixed homogeneous polynomials.

This problem has been studied in various forms over a number of years going back to Littman \cite{littman}, Stein \cite{steinconv} and Strichartz \cite{str}. More closely related results to the present work are found in the work of Ferreyra, Godoy and Urciuolo who give necessary and sufficient conditions for $L^p-L^q$ boundedness of convolution operators associated to homogeneous polynomials \cite{FGU06, U06} and specific cases of mixed homogeneous polynomials \cite{FGU97a, FGU97b, FGU98}, and Iosevich, Sawyer and Seeger \cite{ISS} who give sufficient conditions for $L^p-L^q$ boundedness of convolution operators associated to convex functions. Our aim is to look beyond convex functions. To this end the consideration of mixed homogeneous polynomials is a first step which should lead to further results for general polynomials and real analytic functions. 
Our results do also improve the known results of \cite{ISS} for cases where the mixed homogeneous polynomial happens to be convex.

In parallel to the aforementioned work, there has been work on weighted versions of the problem, where instead of considering the Riemannian surface measure of the hypersurface, one considers the so-called affine invariant surface measure, introducing a weight which is a power of the Gaussian curvature, and one tries to obtain the same $L^p-L^q$ boundedness as for the most well-curved case corresponding to nowhere vanishing Gaussian curvature. This was done by Oberlin \cite{O00}, who proved restricted weak type endpoint estimates, and Gressman \cite{GressJGA}, who proved strong type endpoint estimates. These weighted estimates might be expected to be in some sense stronger and to imply the unweighted ones, as is the case when the underlying manifold is a curve, in both the Fourier restriction and the convolution problems (see e.g. \cite{DS15,DrM1}). However, even though these weighted estimates prove to be extremely useful in our proof, in the case of hypersurfaces they do not contain all the information. This may indeed be expected by making the simple observation that in the homogeneous case there are nontrivial $L^p-L^q$ bounds for operators associated to hypersurfaces whose Gaussian curvature even vanishes identically, hence for which the corresponding weighted operator is the zero operator.

\section{Statements of results}

Let $\kappa = (\kappa_1,\kappa_2)$, $\kappa_1,\kappa_2>0$. A function $g$ on $\R^2$ is called $\kappa$-homogeneous of degree $a$ if $g(r^{\kappa_1} \cdot,r^{\kappa_2} \cdot) = r^a g$ for all $r>0$. Such functions are also called mixed homogeneous. The exponent $a$ is called the $\kappa$-degree of $g$. The function $g$ is called homogeneous if $\kappa_1=\kappa_2$. The convolution operator under consideration is defined, initially for $f\in\mathcal{S}(\R^3)$, by
\begin{equation} \label{defA}
\mathcal{A}f(x)=\int_{\R^2}f(x-\Phi(y))\psi(y)dy,
\end{equation}
where $\psi\in C^\infty_0(\R^2)$ satisfies
\[0\leq\psi\leq1,\quad \supp\psi\subseteq[-1,1]^2,\quad \psi=1\ \text{on}\ [-1/2,1/2]^2.\]
The function $\Phi$ parametrises the associated hypersurface in $\R^3$ and has the form $\Phi(y)=(y_1,y_2, \varphi(y_1,y_2))$, where $\varphi$ is a mixed homogeneous polynomial. We will exclude from our analysis the cases where $\varphi$ is a homogeneous polynomial (in particular when it is a pure monomial of the form $\varphi(y_1,y_2)=cy_1^{A}y_2^{B}$) since these cases have been dealt with in \cite{FGU06, U06} and they are somewhat special (see the discussion directly after Theorem \ref{main}). Hence, from this point onwards, whenever we use the term mixed homogeneous, we will assume without loss of generality that $\kappa_1 < \kappa_2$. In addition, if $\nabla\varphi(0,0)\not=(0,0)$, then after a linear transfomation in the ambient space which does not change the $L^p \to L^q$ norm of $\scriptA$ (cf. Lemma \ref{childrenlemma}), $\varphi$ becomes a monomial and can be treated as in \cite{FGU06}. Therefore in Theorem \ref{main} we only consider the cases with $\nabla\varphi(0,0)=(0,0)$. Note that with this definition, it is not possible that the Hessian determinant of a mixed homogeneous polynomial vanishes identically (as will be shown later in Corollary \ref{flat}). 

We first quote parts of a proposition and its corollary from \cite{IM} (Proposition 2.2 and Corollary 2.3 in that article) which give us further information on the structure of mixed homogeneous polynomials and will aid us in formulating our theorem. Here $\gcd$ stands for the greatest common divisor.
\begin{proposition}[{\bf Ikromov and M\"uller}] \label{hompolynomial}
Let $\varphi\colon\R^2\to\R$ be a $\kappa$-homogeneous polynomial of degree one. 
Assume that $\varphi$ is not of the form $\varphi(y_1,y_2)=cy_1^{A}y_2^{B}$, $c\in\R$, $A,B\in\N_0$. Then $\kappa$ is uniquely determined by the polynomial $\varphi$ and we have $\kappa_1$, $\kappa_2\in\Q$. Furthermore, 
\[ 
\kappa=(\kappa_1,\kappa_2)=\left(\frac{s}{m},\frac{r}{m}\right),\quad \gcd(r,s)=1.
\]
The polynomial $\varphi$ can be factorised as
\begin{equation} \label{factor}
\varphi(y_1,y_2)=Cy_1^{\nu_1}y_2^{\nu_2}\prod\limits_{j=1}^M(y_2^s-\lambda_jy_1^r)^{n_j},
\end{equation}
with $M\in\N$, $\nu_1$, $\nu_2\in \N_0$, distinct $\lambda_j\in\C\setminus\{0\}$ with multiplicities $n_j\in\N$, $j\in\{1,\ldots,M\}$, and $C\in \R\setminus\{0\}$. If we put $n:= \sum_{j=1}^M n_j$, we also have
\[
\frac{1}{\kappa_1 + \kappa_2} = \frac{\nu_1 s + \nu_2 r + nrs}{r+s}.
\]
\end{proposition}

For the rest of this article we fix the notation introduced in Proposition \ref{hompolynomial} and denote the Hessian determinant $\det\varphi''$ by $w$. Without loss of generality (cf. Lemma \ref{childrenlemma}), we may assume that $C=1$. We define, for any polynomial $\varphi$ as in Proposition \ref{hompolynomial}, its {\it homogeneous distance}
\[
d_h(\varphi) := \frac{1}{\kappa_1 + \kappa_2}.
\]
The quantity $d_h(\varphi)$ will be written as $d_h$ when no confusion arises. We also define the {\it height} of $\varphi$ by
\[
h(\varphi):= \max\{d_h(\varphi),\nu_1,\nu_2,\max_{\lambda_j\in\R, 1\le j\le M} n_j\},
\]
in the case where $\varphi$ is not a monomial and by
\[
h(\varphi):= \max\{\nu_1,\nu_2\},
\]
in the case that it is. This definition of the height will be sufficient for our purposes. There is a well-known link with the Newton polyhedron of $\varphi$ in terms of which the height is usually defined, for which we refer the interested reader to \cite{IM} (in particular Corollary 3.4 in that article).

\begin{corollary}[{\bf Ikromov and M\"uller}] \label{hompolynomialcor}
Let $\varphi$ be a $(\kappa_1,\kappa_2)$-homogeneous polynomial of degree one as in Proposition \ref{hompolynomial} and consider the representation \eqref{factor} of $\varphi$.
\begin{itemize}
\item[(a)] If $\kappa_2/\kappa_1 \not\in\N$, i.e. if $s\ge 2$, then $n < d_h(\varphi)$. In particular, every real root $x_2 = \lambda^{1/s}_j x^{r/s}_1$ of $\varphi$ has multiplicity $n_j < d_h(\varphi)$.
\item[(b)] If $\kappa_2/\kappa_1 \in \N$, i.e. if $s = 1$, then there exists at most one real root of $\varphi$ on the unit circle $S^1$ of multiplicity greater than $d_h(\varphi)$. More precisely, if we put $n_0 := \nu_1, n_{M+1} := \nu_2$, choose $j_0 \in \{0,\ldots ,M+1\}$ so that $n_{j_0}= \max \{n_0,\ldots ,n_{M+1}\}$ and assume that $n_{j_0} > d_h(\varphi)$, then $n_j < d_h(\varphi)$ for every $j\not=j_0$.
\end{itemize}
\end{corollary}

In light of the above proposition and its corollary, our main theorem is as follows.
\begin{theorem} \label{main}
Let $\varphi$ be a $(\kappa_1,\kappa_2)$-homogeneous polynomial of degree one, not of the form $\varphi(y_1,y_2)=cy_1^{A}y_2^{B}$, $c\in\R$, $A,B\in\N_0$, and with \[\kappa_1\not=\kappa_2,\quad\nabla\varphi(0,0)=(0,0).\]
Let $N$ be the highest multiplicity of any real root of $\varphi$ that is not along one of the coordinate axes and $T$ be the highest multiplicity of any real root of the Hessian determinant $\det \varphi''$.
Then the operator $\scriptA$, defined by \eqref{defA}, is bounded from $L^p(\R^3)$ to $L^q(\R^3)$ for $p$ and $q$ satisfying the conditions below (here $d_h=d_h(\varphi)$):
\begin{equation} \label{c1}
\frac{1}{q}\leq\frac{1}{p},
\end{equation}
\begin{equation} \label{c2}
\frac{1}{q}\ge \frac{3}{p}-2,
\end{equation}
\begin{equation} \label{c3}
\frac{1}{q}\ge \frac{1}{3p},
\end{equation}
\begin{equation} \label{cdh}
\frac{1}{q}>\frac{1}{p}-\frac{1}{d_h+1},
\end{equation}
and
\begin{itemize}
\item[(a)] if $N\ge d_h + 1/2$,
\begin{equation} \label{c4}
\frac{1}{q}>\frac{1}{p}-\frac{1}{N},
\end{equation}
\begin{equation} \label{c5}
\frac{1}{q}>\frac{N+2}{N+1}\frac{1}{p}-\frac{2}{N+1},
\end{equation}
\begin{equation} \label{c6}
\frac{1}{q}>\frac{N+1}{N+2}\frac{1}{p}-\frac{1}{N+2},
\end{equation}
\item[(b)] if $N<d_h + 1/2$ and $\max\{\nu_1,\nu_2\} \ge d_h$,
\begin{equation} \label{c7}
\frac{1}{q}>\frac{1}{p}-\frac{1}{\max\{\nu_1,\nu_2\} +1},
\end{equation}
\item[(c)] if $N<d_h + 1/2$, $\max\{\nu_1,\nu_2\} < d_h$ and the real root of $\det\varphi''$ with the highest multiplicity is either along one of the axes or coincides with a real root of $\varphi$,
\begin{equation} \label{c9}
\frac{1}{q}>\frac{2T+5}{T+3}\frac{1}{p}-1,
\end{equation}
\begin{equation} \label{c10}
\frac{1}{q}>\frac{T+3}{2T+5}\frac{1}{p}-\frac{1}{2T+5},
\end{equation}
\item[(d)] if $N<d_h + 1/2$, $\max\{\nu_1,\nu_2\} < d_h$ and the real root of $\det\varphi''$ with the highest multiplicity is not along the axes and does not coincide with a real root of $\varphi$,
\begin{equation} \label{c12}
\frac{1}{q}>\frac{5}{3}\frac{1}{p}-\frac{2T+12}{3T+12},
\end{equation}
\begin{equation} \label{c13}
\frac{1}{q}>\frac{3}{5}\frac{1}{p}-\frac{4}{T+4}.
\end{equation}
\end{itemize}
\end{theorem}

Note that the conditions \eqref{c3}, \eqref{c6}, \eqref{c10}, \eqref{c13} are ``dual" to \eqref{c2}, \eqref{c5}, \eqref{c9}, \eqref{c12}, respectively. Also conditions \eqref{c9}, \eqref{c10}, \eqref{c12} and \eqref{c13} are redundant in parts (c) and (d) of Theorem \ref{main} if $T\le 2d_h(\varphi)-2$. This is the case for instance if the real root of $\det\varphi''$ with the highest multiplicity is along one of the axes and $\varphi$ vanishes along that axis as well or if the real root of $\det\varphi''$ with the highest multiplicity is not along the axes and it coincides with a real root of $\varphi$ with multiplicity greater or equal to $2$.

Comparing this theorem with the results of Ferreyra, Godoy and Urciuolo \cite{FGU06, U06} for the homogeneous case, one observes that the $L^p-L^q$ boundedness regions corresponding to homogeneous polynomials are not obtained simply by substituting $\kappa_1=\kappa_2$ in the mixed homogeneous case. The latter regions are in general larger than the former. 

In the case where the hypersurface is also convex, Iosevich, Sawyer and Seeger \cite{ISS} have made a conjecture (they have actually made a more general conjecture for Sobolev spaces) that describes the $L^p-L^q$ boundedness region in terms of the ``multitype", which they define. Comparing the examples $\varphi(y_1,y_2)=y_2^4+y_1^{12}$ and $\varphi(y_1,y_2)=y_2^4+y_2^2y_1^6-y_2y_1^9+y_1^{12}$ one may observe that they are both convex and have the same multitype, hence those authors obtain the same $L^p-L^q$ boundedness region for both. However, using the notation of Theorem \ref{main}, we have $d_h=3$ in both examples, $T=10$ for the first example and $T=4$ for the second, hence we obtain a larger region for the second example than the conjectured one in \cite{ISS}. The $L^p-L^q$ bounedness region for the first example, dictated by Theorem \ref{main}, is the same one as in \cite{FGU98, ISS} and it is sharp. These examples show that, at least for $L^p-L^q$ boundedness, the Iosevich, Sawyer and Seeger \cite{ISS} conjecture is not appropriate.

In Section \ref{neccon} we provide necessary conditions for parts (a), (b) and (c) (in certain cases, see Lemma \ref{Lemma 3.4}). The sharpness of the conditions in the remaining cases of parts (c) and (d) is not known, nor are any examples, that show that the conditions in (c) and (d) are not sharp, known to the authors at this stage.

We organise our proof in three parts: $N\ge d_h + 1, d_h+1>N\ge d_h+1/2$, and $N<d_h+1/2$; the last part splitting further. In each part, we exploit the root structure of $w=\det\varphi''$. After performing dyadic decompositions and rescaling, the key step is to identify two $L^p\to L^{p'}$ estimates ($p'$ is the dual exponent to $p$) on the dyadic pieces after which the rest of the proof follows naturally by interpolation on the pieces and summation. One of these two estimates is an $L^{4/3}\to L^4$ estimate, which is in fact a direct consequence of the weighted $L^{4/3}\to L^4$ estimate for convolution with hypersurfaces of Gressman \cite{GressJGA}. The other key estimate is mostly an $L^{3/2}\to L^3$ estimate, except in the last case where an $L^{8/5}\to L^{8/3}$ estimate is used, which we obtain using decay estimates for certain oscillatory integrals and which depends on the root structure of $\varphi$ itself. 

Weighted boundedness estimates from $L^{3/2}$ to $L^3$ for the operators in question also appear in Gressman \cite{GressJGA} as a consequence of Theorem 6 in that article. However, that theorem seems to be insufficient for our purposes and therefore not to be optimal in the way that the weighted $L^{4/3}\to L^4$ estimates are. This can be demonstrated by looking at the simple example $\varphi(y_1,y_2)=(y_2-y_1^2)^2$. The required estimate in this case amounts to satisfying the condition of Theorem 6 in \cite{GressJGA} with the weight $w(x,y)=|x_2-y_2-(x_1-y_1)^2|^{\epsilon -1}$, for all $\epsilon>0$. In turn, if that condition were satisfied, it would imply that, for all $\epsilon>0$ and some $a,b,\alpha\in\R$,
\begin{eqnarray*}
\infty &>& \sup_{\lambda\in\R} |\lambda|\int_{0\le y_1\le \frac12} \int_{y_1^2\le y_2\le \frac12 +y_1^2} \frac{dy}{(y_2+y_1^2)^{1-\epsilon} [1+|\lambda (2a+4by_1)(y_2-y_1^2)|^\alpha]} \\
&\gtrsim& \sup_{\lambda\in\R} |\lambda|\int_{0\le y_1\le \frac12} \int_{y_1^2\le y_2\le \frac12 +y_1^2} \frac{1}{(y_2+y_1^2)^{1-\epsilon} [1+|\lambda (y_2-y_1^2)|^\alpha]} dy \\
&=& \sup_{\lambda\in\R} |\lambda|^{1-\epsilon} \int_0^{\frac12} \int_0^{\frac{|\lambda|}{2}}  \frac{1}{(y_2+2(|\lambda|^{\frac12}y_1)^2)^{1-\epsilon} (1+y_2^\alpha)} dy \\
&\ge& \sup_{\lambda >4} \lambda^{1-\epsilon} \int_0^{\frac{1}{\lambda^{1/2}}} \int_{\frac12}^1  \frac{1}{(y_2+2(\lambda^{\frac12}y_1)^2)^{1-\epsilon} (1+y_2^\alpha)} dy \\
&\gtrsim& \sup_{\lambda >4} \lambda^{\frac12-\epsilon},
\end{eqnarray*}
which is not true for $\epsilon<1/2$. 

In addition to the notation introduced above, we need the following. For a mixed homogeneous polynomial 
\[\varphi(x_1,x_2)=\Sum\limits_{j=0}^J\Sum_{k=0}^K a_{jk}x_1^jx_2^k,\] 
we denote by 
\[\mathcal T(\varphi)=\left\{(j,k)\in\N_0^2:\ a_{jk}\neq0\right\}\]
the Taylor support of $\varphi$ at the origin. Finally, if $A$ and $B$ are two nonnegative quantities, we write $A \lesssim B$ if $A \leq cB$ for some constant $c$, which will be allowed to change from line to line. By $A\sim B$ we mean $A \lesssim B \lesssim A$. 


We will also need the following lemma. The proof is very easy and is therefore omitted.
\begin{lemma}\label{childrenlemma}
Let $S$ be a linear operator mapping from the space of Schwartz functions to the space of all measurable functions and let $A\in\operatorname{GL}(n;\R)$. For $f\in\mathcal{S}$ and $x\in\R^n$, let $S_Af(x) = S(f\circ A)(A^{-1}x)$.
Then \[\|S_A\|_{L^p\to L^q}=|\det A|^{\frac{1}{q}-\frac{1}{p}}\|S\|_{L^p\to L^q}.\] 
\end{lemma}

We start in Section \ref{neccon} with some necessary conditions for $L^p-L^q$ boundedness and we continue in Section \ref{dimline} with a simple interpolation argument that gives estimates on the line $1/q = 3/p -2$. Sections \ref{N>d+1} and \ref{inter} deal with part (a) of Theorem \ref{main}, Section \ref{nu} deals with part (b), Sections \ref{s>1} and \ref{s=1} deal with parts (c) and (d) for the cases where $s\ge 2$ and $s=1$, respectively. As mentioned above, the proofs in all sections follow the pattern of performing a bidyadic decomposition, using oscillatory integral estimates and the weighted estimate for convolution with hypersurfaces of Gressman \cite{GressJGA} on each piece, interpolation on each piece and summing up the pieces. We conclude in the final section with a proof of a lemma that is used in Sections \ref{N>d+1} and \ref{inter} and a discussion of the relation between the height of the mixed homogeneous polynomial $\varphi$ and the height of its Hessian determinant $w$.

\section{Necessary conditions}
\label{neccon}
We give here some necessary conditions.
\begin{lemma} Suppose $\scriptA$ is bounded from $L^p(\R^3)$ to $L^q(\R^3)$. Then \eqref{c1}, \eqref{c2}, \eqref{c3} hold.
\end{lemma}
\begin{proof}
To show \eqref{c1}, let $f_N={\1}_{[-2N,2N]^3}$ for large $N$. If $x\in[-N,N]^3$, then $x-\Phi(y)\in[-2N,2N]^3$ for any $y\in[-1/2,1/2]^2$ and so $Tf_N(x)\gtrsim 1$ for any such $x$. Hence
\[N^{\frac{3}{q}} \lsm \|\scriptA f_N\|_{L^q(\R^3)}\lsm \|f_N\|_{L^p(\R^3)} \sim N^\frac{3}{p},\] which in turn implies \eqref{c1}.

To show \eqref{c2} we consider for $\delta$ small the set 
\[
X=\{x\in\R^3: (x_1,x_2)\in[-1/4,1/4]^2, |x_3-\varphi(x_1,x_2)|\leq\delta/2\}
\] 
and $f_{\delta}=\1_{[-\delta,\delta]^3}$. Clearly $|X|\sim\delta$. For any $x\in X$, let 
\[
Y_x=\left\{y\in [-1/2,1/2]^2: |y-(x_1,x_2)| \leq \frac{\delta}{2+2\|\nabla\varphi\|_{L^\infty([-1,1]^2)}}\right\}.
\]
Then for any sufficiently small $\delta$, any $x\in X$ and $y\in Y_x$ we have $x-\Phi(y)\in[-\delta,\delta]^3$ and so $Tf_\delta(x)\gtrsim \delta^2$. Thus 
\[
\delta^\frac{1}{q} \delta^2 \lsm \|\scriptA f_\delta\|_{L^q(\R^3)}\lsm \|f_\delta\|_{L^p(\R^3)} \sim \delta^{\frac{3}{p}},
\]
implying \eqref{c2}. Inequality \eqref{c3} follows by duality, since $\scriptA$ is essentially self-adjoint.
\end{proof}

\begin{lemma} Suppose $\scriptA$ is bounded from $L^p(\R^3)$ to $L^q(\R^3)$. Then
\begin{equation} \label{Cdis}
\quad \frac{1}{q}\geq\frac{1}{p}-\frac{1}{\max\{\nu_1,\nu_2,d_h\}+1}.
\end{equation}
\end{lemma}
\begin{proof}
We first show that $\frac1q \geq \frac1p - \frac{1}{\nu_2+1}$ assuming $\nu_2\ge 1$, otherwise there is nothing to prove. The condition $\frac1q \geq \frac1p - \frac{1}{\nu_1+1}$ can be shown in a very similar way. We factorise $\varphi(y) = y_2^{\nu_2} P(y)$ and let 
\[
f_\delta=\1_{[-2,2]\times[-2\delta^{\frac{1}{\nu_2}},2\delta^{\frac{1}{\nu_2}}]\times[-K\delta,K\delta]}, 
\]
where $K=1+\|P\|_{L^\infty([-1,1]^2)}$. Then any sufficiently small $\delta$, for any $x\in[-1,1]\times[-\delta^{1/\nu_2},\delta^{1/\nu_2}]\times[-\delta,\delta]$ and any $y\in[-1/2,1/2]\times[-\delta^{1/\nu_2},\delta^{1/\nu_2}]$ we have $x-\Phi(y)\in\supp{f_\delta}$. Thus
\[
(\delta^{1+\frac{1}{\nu_2}})^{\frac1q} \delta^{\frac{1}{\nu_2}}\lsm\|\scriptA f_\delta\|_{L^q(\R^3)}\lsm \|f_\delta\|_{L^p(\R^3)} \sim (\delta^{1+\frac{1}{\nu_2}})^{\frac1p},
\] 
which gives the desired result. 

To show that $\frac1q \geq \frac1p - \frac{1}{d_h+1}$ we let
\[
f_\delta=\1_{[-2\delta^{\kappa_1},2\delta^{\kappa_1}]\times[-2\delta^{\kappa_2},2\delta^{\kappa_2}]\times[-(M+1)\delta,(M+1)\delta]},
\] 
where $M=\sum\limits_{\alpha,\beta}|a_{\alpha \beta}|$ and $a_{\alpha \beta}$ are the coefficients in the Taylor expansion of $\varphi$ around $0$. Then any sufficiently small $\delta$, for any 
$x\in[-\delta^{\kappa_1},\delta^{\kappa_1}]\times[-\delta^{\kappa_2},\delta^{\kappa_2}]\times[-\delta,\delta]$ and any $y\in[-\delta^{\kappa_1},\delta^{\kappa_1}]\times[-\delta^{\kappa_2},\delta^{\kappa_2}]$ we have $x-\Phi(y)\in\supp{f_\delta}$. Hence
\[
(\delta^{1+\kappa_1+\kappa_2})^{\frac1q}\delta^{\kappa_1+\kappa_2}\lsm\|\scriptA f_\delta\|_{L^q(\R^3)}\lsm \|f_\delta\|_{L^p(\R^3)}\sim
(\delta^{1+\kappa_1+\kappa_2})^{\frac{1}{p}},
\] 
which gives 
\[\frac{1}{q}\geq\frac{1}{p}-\frac{\kappa_1+\kappa_2}{\kappa_1+\kappa_2+1}=\frac{1}{p}-\frac{1}{d_h+1},
\]
concluding the proof of the lemma.
\end{proof}
Condition \eqref{Cdis} reduces to conditions \eqref{cdh}, \eqref{c7} and partly \eqref{c4} in the various cases of Theorem \ref{main}, with strict inequalities replacing the inequality in \eqref{Cdis}.

\begin{lemma}
Suppose $\varphi(y_1,y_2)=(y_2-\lambda y_1^r)^NP(y_1,y_2)$ for some $\lambda\in\R\setminus\{0\}$, $N\in\N$ and some polynomial $P$, and $\scriptA$ is bounded from $L^p(\R^3)$ to $L^q(\R^3)$. Then 
\begin{equation} \label{N1}
\frac{1}{q}\geq\frac{1}{p}-\frac{1}{N},
\end{equation}
\begin{equation} \label{N2}
\frac{1}{q}\geq\frac{N+2}{N+1}\frac{1}{p}-\frac{2}{N+1}.
\end{equation}
\begin{equation} \label{N3}
\frac{1}{q}\ge \frac{N+1}{N+2}\frac{1}{p}-\frac{1}{N+2},
\end{equation}
\end{lemma}
\begin{proof}Let $M=\|P\|_{L^\infty([-2,2]^2)}$. To show \eqref{N1} we let 
\[
f_\delta=\1_{[-2,2]^2\times[-(M+1)\delta^N,(M+1)\delta^N]}
\]
and
\[
Y=\{y\in [-1/2,1/2]^2:\ y_1\in[0,(|\lambda|+1)^{-\frac{1}{r}}],\ |y_2-\lambda y_1^r|\leq\delta\}.
\]
For any $(x_1,x_2,x_3)\in[-1,1]^2\times[-\delta^N,\delta^N]$ and $y \in Y$ we have $x-\Phi(y)\in\supp f_\delta$. Thus
\[
\delta^\frac{N}{q} \delta \lsm \|\scriptA f_\delta\|_{L^q(\R^3)} \lsm \|f_\delta\|_{L^p(\R^3)} \sim \delta^\frac{N}{p},
\] 
for all sufficiently small $\delta$, implying \eqref{N1}.

To show \eqref{N2}, let
\[
f_\delta=\1_{\{z\in\R^3:\ |z_1|\leq\delta,\ |z_2|\leq 2(1+|\lambda|r)\delta,\ |z_3|\leq(M+1)\delta^N\}}.
\]
For $x\in\R^3$ with $|x_1|\leq 1$, $|x_2-\lambda x_1^r|\leq\delta$, $|x_3|\leq\delta^N$ consider 
\[
Y_x=\{y\in [-1/2,1/2]^2:\ y_1\in[x_1-\delta,x_1], |y_2-\lambda y_1^r|\leq\delta\}.
\]
Observe that for all sufficiently small $\delta$ and $y\in Y_x$, we have $x-\Phi(y)\in \supp f_\delta$ and hence
\[
\delta^\frac{N+1}{q} \delta^2 \lsm \|\scriptA f_\delta\|_{L^q(\R^3)}\lsm \|f_\delta\|_{L^p(\R^3)} \sim \delta^\frac{N+2}{p}
\]
proving \eqref{N2}. Condition \eqref{N3} follows by duality.
\end{proof}

The last lemma in this section shows that the conditions of part (c) of Theorem \ref{main} are necessary in certain cases, namely those dealt with in Section \ref{s>1} and part of Section \ref{s=1}, where the root of the Hessian determinant with the highest multiplicity is the $y_2$-axis.

\begin{lemma} \label{Lemma 3.4}
Let $(A,B)\in\mathcal{T}(\varphi)$ be such that if $(\alpha,\beta)\in\mathcal{T}(\varphi)$ and $\alpha<A$, then $\alpha=0$. If $\scriptA$ is bounded from $L^p(\R^3)$ to $L^q(\R^3)$, then
\begin{equation} \label{ML1}
\frac1q\ge\frac{2A+1}{A+1}\frac1p-1
\end{equation}
and
\begin{equation} \label{ML2}
\frac{1}{q}\ge\frac{A+1}{2A+1}\frac{1}{p}-\frac{1}{2A+1}.
\end{equation}
\end{lemma}
\begin{proof}
To show \eqref{ML1}, let
\[
f_\delta=\1_{[-2\delta^{1/A},2\delta^{1/A}]\times [-\delta,\delta] \times [-K\delta,K\delta]},
\]
for some sufficiently large $K$. For $x\in\R^3$ with $|x_1|\leq \delta^{1/A}$, $x_2\in [0,1/4]$, $|x_3 - \varphi(x_1,x_2)|\leq\delta$ consider 
\[
Y_x=\{y\in [-1/2,1/2]^2:\ y_1\in[-\delta^{1/A},\delta^{1/A}], |y_2-x_2|\leq\delta\}.
\]
Observe that $|Y_x| \sim \delta^\frac{A+1}{A}$ and for all sufficiently small $\delta$ and $y\in Y_x$, we have
\begin{eqnarray*}
|x_3-\varphi(y_1,y_2)| &\le& |x_3-\varphi(x_1,x_2)|+|\varphi(x_1,x_2)-\varphi(y_1,y_2)|\\
&\le& \delta + |a_{0 F}||x_2^F-y_2^F| + \sum_{(\alpha,\beta) \in\mathcal{T}(\varphi), \alpha\ge A} a_{\alpha \beta} (|x_1^\alpha x_2^\beta| + |y_1^\alpha y_2^\beta|) \\
&\lsm& \delta,
\end{eqnarray*}
where $a_{0 F}\not=0$ if $(0,F)\in\mathcal{T}(\varphi)$. Hence for this choice of $x,y$, we have $x-\Phi(y)\in \supp f_\delta$ and thus
\[
\delta^\frac{A+1}{Aq} \delta^{\frac{A+1}{A}} \lsm \|\scriptA f_\delta\|_{L^q(\R^3)}\lsm \|f_\delta\|_{L^p(\R^3)} \sim \delta^\frac{2A+1}{Ap}
\]
proving \eqref{ML1}. Condition \eqref{ML2} follows by duality.
\end{proof}

\section{A preliminary estimate}
\label{dimline}
Let $H$ be the height of the mixed homogeneous polynomial $w$. We show here that $\scriptA$ is bounded from $L^p(\R^3)$ to $L^q(\R^3)$ for all $(\frac{1}{p},\frac{1}{q})$ satisfying $\frac{1}{q}=\frac{3}{p}-2$ and $\frac{H+3}{H+4}<\frac{1}{p}$. For $\varepsilon>0$ let
\[
\scriptA^\varepsilon f(x)=\int_{\R^2}f(x-\Phi(y))\psi(y)|w(y)|^{-\frac{1}{H}+\varepsilon \gamma}dy,
\]
where $\gamma=\frac{1}{H}+\frac{1}{4}$. Using the strong type endpoint estimate of Gressman (Theorem 3 in \cite{GressJGA}) we obtain $\|\scriptA^1f\|_{L^{4}(\R^3)}\lesssim \|f\|_{L^{\frac43}(\R^3)}$. On the other hand, changing variables and using the mixed homogeneity of $w$, we see that 
\begin{equation} \label{pramanik}
\int_{\R^2}\psi(y)|w(y)|^{-\frac{1}{H}+\varepsilon \gamma}dy<\infty
\end{equation}
for every $\varepsilon>0$, which clearly implies that $\|\scriptA^\varepsilon f\|_{L^{1}(\R^3)}\lesssim \|f\|_{L^{1}(\R^3)}$ (\eqref{pramanik} also generalises to real-analytic functions, see Pramanik \cite{P01}). By analytic interpolation (see e.g. \cite{steinweiss}) we conclude that $\scriptA^{\varepsilon(1-\theta)+\theta}$ is bounded from $L^{p_\theta}(\R^3)$ to $L^{q_\theta}(\R^3)$ for 
\[
\frac{1}{p_\theta}= 1-\theta +\frac34\theta,\quad \frac{1}{q_\theta}= 1-\theta +\frac{\theta}{4}.
\] 
Observe that for $\theta=\frac{4-\varepsilon(4+H)}{(4+H)(1-\varepsilon)}$ we have $\scriptA^{\varepsilon(1-\theta)+\theta}=\scriptA$, i.e. $\scriptA$ is bounded from $L^{p(\varepsilon)}(\R^3)$ to $L^{q(\varepsilon)}(\R^3)$, where $\frac{1}{p(\varepsilon)}=1-\frac{1}{(4+H)(1-\varepsilon)}+\frac{\varepsilon}{4(1-\varepsilon)}$ and $\frac{1}{q(\varepsilon)}=1-\frac{3}{(4+H)(1-\varepsilon)}+\frac{3}{4}\frac{\varepsilon}{1-\varepsilon}$. Note that $(\frac{1}{p(\varepsilon)},\frac{1}{q(\varepsilon)})\longrightarrow(\frac{H+3}{H+4},\frac{H+1}{H+4})$ as $\varepsilon\longrightarrow 0$.

\section{$N\ge d_h(\varphi)+1$}
\label{N>d+1}

Here by Corollary \ref{hompolynomialcor} $s=1$, $N=n_l$, for some $1\le l \le M$, and this is the only multiplicity of a real root greater than $d_h(\varphi)$ hence $\h=N$. We write $\lambda = \lambda_l$. We observe that $\di\geq2$. This is seen as follows. Since $\nabla\varphi(0,0)=(0,0)$ implies $\di\geq1$, we first observe that\[\di=1\Longrightarrow\varphi(y_1,y_2)=ay_1y_2+by_1^{r+1},\quad a\neq0, b\in\R,\] for which $N\leq1$, which is a contradiction. On the other hand, $\di>1\Longrightarrow N\geq3$, since $N$ is an integer. We obtain a contradiction by \[2>\di=\frac{\nu_1+r\nu_2+rn}{r+1}\geq\frac{rn}{r+1}\geq \frac{3r}{r+1}\geq \frac{3\cdot2}{2+1}=2.\]Therefore $\di\geq2$ and $ N\geq3$.  
Lemma \ref{hardcomputation} shows that $w(y)=(y_2-\lambda y_1^r)^{2N-3}\tilde{Q}(y)$, where $\tilde{Q}(y_1,\lambda y_1^r)\neq0$ for $y_1\neq0$.

Since $w$ is $\Big(\frac{\kappa_1}{2(1-\kappa_1-\kappa_2)},\frac{\kappa_2}{2(1-\kappa_1-\kappa_2)}\Big)$-homogeneous of degree one and 
\[N>\di+1/2 \Longleftrightarrow 2N-3>d_h(w),\] we can conclude that the multiplicity of any other real root of $w$ is bounded by $d_h(w)=2\di-2$. In particular, this implies that the height of $w$ is $H=2N-3$. Let $\eta\in C^\infty(\R)$ with \[\supp\eta\subseteq[-2,2], \quad 0\leq\eta\leq1, \quad \eta=1\ \text{on} \ [-1,1],\] and let $\varepsilon>0$ be small. Let 
\[\mathcal{A}_{\mathcal{C}_\lambda}f(x)=\int_{\R^2}f(x-\Phi(y))\psi(y)\eta\Big(\frac{y_2-\lambda y_1^r}{\varepsilon y_1^r}\Big)dy,\]
where $\lambda=\lambda_l$. Let 
\[\mathcal{A}_{{\mathcal{C}_\lambda}^{c}}f(x)=\int_{\R^2}f(x-\Phi(y))\psi(y)\Big(1-\eta\Big(\frac{y_2-\lambda y_1^r}{\varepsilon y_1^r}\Big)\Big)dy.\]
The operator $\mathcal{A}_{{\mathcal{C}_\lambda}^{c}}$ is bounded on a bigger region than $\scriptA_{\mathcal{C}_\lambda}$. In fact, it is bounded on the trapezium given by the lines $\frac{1}{p}\leq\frac{1}{q}$, $\frac{1}{q}\geq\frac3p-2$, $\frac1q \ge \frac{1}{3p}$ and $\frac{1}{q} >\frac1p-\frac{1}{\di+1}$. This is seen as follows. If $\varepsilon>0$ is chosen sufficiently small, then for any $\delta>0$ we have
\[\int_{\R^2}|w(y)|^{-\frac{1}{2\di-2}+\delta}\psi(y)\Big(1-\eta\Big(\frac{y_2-\lambda y_1^r}{\varepsilon y_1^r}\Big)\Big)dy<\infty.\]  The operator \[f\longmapsto\int_{\R^2}f(\cdot-\Phi(y))|w(y)|^{\frac{1}{4}}\psi(y)\Big(1-\eta\Big(\frac{y_2-\lambda y_1^r}{\varepsilon y_1^r}\Big)\Big)dy\]
is bounded as an operator from $L^{\frac{4}{3}}(\R^3)$ to $L^{4}(\R^3)$. A simple interpolation argument as in Section \ref{dimline} yields the desired result. Next, we turn our attention to 
$\mathcal{A}_{\mathcal{C}_\lambda}$.
Change of variables gives
\[\mathcal{A}_{\mathcal{C}_\lambda}f(x)=\int_{\R^2}f(x-\Phi(y_1,y_2+\lambda y_1^r))\psi(y_1,y_2+\lambda y_1^r)\eta\Big(\frac{y_2}{\varepsilon y_1^r}\Big)dy,\]
In this case we decompose bidyadically. For this purpose consider a dyadic partition of unity
\[\Sum\limits_{k=L}^\infty\chi_k(s)=1,\quad\text{for}\ s\in[-2^{-L},0)\cup(0,2^{-L}],\quad L\in\Z. \]
The function $\chi$ is a smooth positive function supported in $[-2,-\frac12]\cup[\frac12,2]$ and we set $\chi_k=\chi(2^k\cdot)$. Then for some $L=L(\lambda,r)$ we obtain
\begin{align*}
&\scriptA_{\mathcal{C}_\lambda}f(x) \ \le \ \Sum\limits_{k=L}^\infty\Sum\limits_{j=0}^\infty\int_{\R^2}f(x-\Phi(y_1,y_2+\lambda y_1^r))\eta\Big(\frac{y_2}{\varepsilon y_1^r}\Big)\chi_j(y_1)\chi_k(y_2)dy\\
&= \ \Sum\limits_{k=L}^\infty\Sum\limits_{j=0}^\infty 2^{-j-k}\int_{\R^2}f(x-\Phi(2^{-j}y_1,2^{-k}y_2+\lambda2^{-jr}y_1^r))\eta\Big(\frac{2^{-k}y_2}{\varepsilon 2^{-jr}y_1^r}\Big)\chi\otimes\chi(y)dy\\
&\le \ \Sum\limits_{k=L}^\infty\Sum\limits_{0\leq j\ll k/r} 2^{-j-k}\int_{\R^2}f(x-\Phi(2^{-j}y_1,2^{-k}y_2+\lambda 2^{-jr}y_1^r))\chi\otimes\chi(y) dy,
\end{align*} 
where $\chi\otimes\chi(y)=\chi(y_1)\chi(y_2)$. We remark that $j$ is indeed much smaller than $k/r$ and this can be achieved assuming $\varepsilon$ to be sufficiently small. Observe that \[\supp\chi\otimes\chi\subseteq\{y_1:\ 2^{-1}\leq|y_1|\leq 2\}\times \{y_2:\ 2^{-1}\leq|y_2|\leq 2\}.\] 
We have \[\Phi(2^{-j}y_1,2^{-k}y_2+\lambda2^{-j}y_1^r)=\begin{pmatrix}
 2^{-j}y_1 \\
 2^{-k}y_2+\lambda2^{-jr}y_1^r\\
2^{-j\nu_1-kn_l-jr\nu_2-jr(n-n_l)}\varphi^{l}_{j,k}(y))
\end{pmatrix}^{tr},\]
where \[\varphi^{l}_{j,k}(y)=y_1^{\nu_1}y_2^{n_l}(2^{-k+jr}y_2+\lambda_ly_1^r)^{\nu_2}\prod\limits_{\genfrac{}{}{0pt}{}{i=1}{i\neq l}}^M(2^{-k+jr}y_2-(\lambda_i-\lambda_l)y_1^r)^{n_i}.\]
Recall that $n_l=N$. We conclude
\[\mathcal{A}_{\mathcal{C}_\lambda}f(x)\leq\Sum\limits_{k=L}^\infty\Sum\limits_{0\leq j \ll k/r}2^{-j-k}\underbrace{\mathcal{A}_{j,k}^l(f\circ D_{j,k})( D_{j,k}^{-1}x)}_{\tilde{\mathcal{A}}^l_{j,k}f(x)},\]where
\[\mathcal{A}^l_{j,k}f(x)=\int_{\R^2}f(x-\Phi^l_{j,k}(y))\chi\otimes\chi(y)dy=:f\ast\mu_{j,k}^l(x)\]
and\[\Phi^l_{j,k}(y)=(y_1,2^{-k+jr}y_2+\lambda_ly_1^{r},\varphi^{l}_{j,k}(y)),\]
\[D_{j,k}(z_1,z_2,z_3)=(2^{-j}z_1,2^{-jr}z_2,2^{-j\nu_1-jr(\nu_2+n-N)}2^{-kN}).\]
\begin{lemma} \label{decaylemma} We have
\begin{equation}\label{decay1}|\widehat{\mu^l_{j,k}}(\xi)|\lesssim(1+|\xi|)^{-\frac12}.\end{equation}
\end{lemma}
\begin{proof}
Let $\delta_{j,k}=2^{-k+jr}$. Clearly,
\[\widehat{\mu^l_{j,k}}(\xi_1,\xi_2,\xi_3)=\int_{\R^2}e^{-i(\xi_1y_1+\xi_2(\delta_{j,k}y_2+\lambda_ly_1^{r})+\xi_3\varphi^{l}_{j,k}(y))}\chi\otimes\chi(y)dy.\]

We can clearly assume that $|\xi|\geq2$. If $|\xi_1|\gg|\xi_2|+|\xi_3|$, then $|\xi_1|\sim|\xi|$. Integration by parts in the $y_1$ variable then gives even a better estimate
\[|\widehat{\mu^l_{j,k}}(\xi)|\lesssim\frac{1}{|\xi_1|}\sim\frac{1}{1+|\xi|}.\]
We can therefore assume that $|\xi_1|\lesssim|\xi_2|+|\xi_3|$.

\vspace{1em}
\noindent\textit{Case 1: $|\xi_2|\gg|\xi_3|$.} In this case $0\neq|\xi_2|\sim|\xi|$ and $\frac{|\xi_3|}{|\xi_2|}\varphi^{l}_{j,k}(y)$ is small on $\supp \chi\otimes\chi$ in any $C^K$-norm in $(y_1,y_2)$. Since $r\geq2$, we have
\[|\partial_{y_1}^2\lambda_ly_1^r|\gtrsim1.\]
Therefore we can apply van der Corput's Lemma in the $y_1$ variable and obtain 
\[|\widehat{\mu^l_{j,k}}(\xi)|\lesssim\frac{1}{|\xi_2|^{\frac{1}{2}}}\sim\frac{1}{(1+|\xi|)^{\frac12}}.\]

\vspace{1em}
\noindent\textit{Case 2: $|\xi_3|\gtrsim|\xi_2|$.} This gives $0\neq|\xi_3|\sim|\xi|$. If $\delta_{j,k}$ is sufficiently small, then $|\partial_{y_2}^2\varphi^{l}_{j,k}(y)|\gtrsim1$, since $N\geq2$ (recall that $N\geq3$). If we apply van der Corput's Lemma in $y_2$ we obtain
\[|\widehat{\mu^l_{j,k}}(\xi)|\lesssim\frac{1}{|\xi_3|^{\frac{1}{2}}}\sim\frac{1}{(1+|\xi|)^{\frac12}}.\]
This concludes the proof of \eqref{decay1}.
\end{proof}
It is well-known (see \cite{littman, str}) that if we convolve with a measure whose Fourier transform decays to the order $(1+|\xi|)^{-\rho}$, then the corresponding operator is bounded from $L^p$ to $L^{p'}$ for 
\[
\frac1p= \frac12 \left(1+\frac{\rho}{\rho+1}\right).
\]
Hence estimate \eqref{decay1} gives \[\|\mathcal{A}^l_{j,k}\|_{L^{\frac32}(\R^{3})\to L^{3}(\R^3)}\lesssim1.\] Together with Lemma \ref{childrenlemma}, this implies that 
\[\|\tilde{\mathcal{A}}^l_{j,k}\|_{L^{\frac32}(\R^{3})\to L^{3}(\R^3)}\lesssim 2^{j(\frac{r+1}{3})+j(\frac{\nu_1+r\nu_2+r(n-N)}{3})+\frac{kN}{3}}.\] Estimate 
$\|\tilde{\mathcal{A}}^l_{j,k}\|_{L^{1}(\R^{3})\to L^{1}(\R^3)}\lesssim1$ and complex interpolation then give  
\[\|\tilde{\mathcal{A}}^l_{j,k}\|_{L^{p_\theta}(\R^3)\to L^{q_\theta}(\R^3)}\lesssim 2^{\frac{kN}{3}\theta}2^{j(\frac{r+1}{3}+\frac{\nu_1+r\nu_2+r(n-N)}{3})\theta},\] where \[\frac{1}{p_{\theta}}=1-\theta+\frac{2}{3}\theta \quad\text{and}\quad\frac{1}{q_\theta}=1-\theta+\frac{1}{3}\theta.\]
We conclude that \[\|\mathcal{A}_{\mathcal{C}_\lambda}\|_{L^{p_\theta}(\R^3)\to L^{q_\theta}(\R^3)}\lesssim\Sum\limits_{k=L}^\infty 2^{-k+\frac{kN}{3}\theta}\Sum\limits_{j\leq\frac{k}{r}}2^{-j+j(\frac{r+1}{3}+\frac{\nu_1+r\nu_2+r(n-N)}{3})\theta}.\]
It is enough to show that 
\begin{equation}\label{sum1}\Sum\limits_{k=0}^\infty 2^{-k+\frac{kN}{3}\theta}\Sum\limits_{j\leq\frac{k}{r}}2^{-j+j(\frac{r+1}{3}+\frac{\nu_1+r\nu_2+r(n-N)}{3})\theta}<\infty\end{equation}
for any $\theta \in(0,\frac{3}{N})$. Recall that since $N\geq3$, we have indeed $(0,\frac{3}{N})\subseteq(0,1)$.
Then for each $\theta \in(0,\frac{3}{N})\subseteq(0,1)$ we have
\begin{equation*}\renewcommand{\arraystretch}{1.8}\begin{array}{rcl}
-1+\theta(\frac{r+1+\nu_1+r\nu_2+r(n-N)}{3})&<&-1+\frac{3}{N}(\frac{r+1+\nu_1+r\nu_2+r(n-N)}{3})\\
&=&-1+\frac{r+1+(r+1)\di-rN}{N}.
\end{array}\end{equation*}
 The last expression is lower or equal to zero if and only if $N\geq \di + 1$. Therefore we have for the inner sum 
\[\Sum\limits_{j\leq\frac{k}{r}}2^{-j+j(\frac{r+1}{3}+\frac{\nu_1+r\nu_2+r(n-N)}{3})\theta}\lesssim1\quad \text{for any} \ \theta\in\Big(0,\frac{3}{N}\Big),\]
and so
\[\Sum\limits_{k=0}^\infty 2^{-k+\frac{kN}{3}\theta}\Sum\limits_{j\leq\frac{k}{r}}2^{-j+j(\frac{r+1}{3}+\frac{\nu_1+r\nu_2+r(n-N)}{3})}\lesssim \Sum\limits_{k=0}^\infty 2^{-k+\frac{kN}{3}\theta}<\infty.\]This shows that the sum in \eqref{sum1} converges for every $\theta \in(0,\frac{3}{N})$.

Observe that, as $\theta\to \frac3N$, $(\frac{1}{p_\theta},\frac{1}{q_\theta})\to (1-\frac1N,1-\frac2N)$, which is the point of intersection of the lines $\frac{1}{q}=\frac{1}{p}-\frac{1}{N}$ and $\frac{1}{q}=\frac{N+2}{N+1}\frac{1}{p}-\frac{2}{N+1}$. The rest of the proof of part (a) of Theorem \ref{main} in the case $N\ge d_h(\varphi) +1$ then follows by the argument of Section \ref{dimline}, using that the height of $w$ is $2N-3$, duality and interpolation.


\section{$\di(\varphi)+1/2\leq N<\di(\varphi)+1$} 
\label{inter}
Although this case might look quite specific, it turns that there are `many' mixed-homogeneous polynomials satisfying this condition. To see this first observe that 
\begin{equation*}\renewcommand{\arraystretch}{1.8}\begin{array}{rcl}
&&\di+1/2\leq N<\di+1 \\
&\Longleftrightarrow&\frac{\nu_1+r\nu_2+rn}{r+1}+1/2\leq N < \frac{\nu_1+r\nu_2+rn}{r+1}+1\\
&\Longleftrightarrow&N\in \left[\nu_1+r\nu_2+r(n-N)+\frac{r+1}{2}, \nu_1+r\nu_2+r(n-N)+r+1\right).\\
\end{array}\end{equation*}
If we let without loss of generality $n_1=N$, then for any given numbers $\nu_1$, $\nu_2$, $r$, $n_2,\ldots,n_M\in \N$ with \[\nu_1+r\nu_2+r\sum\limits_{i=2}^Mn_i+r+1>N\geq \nu_1+r\nu_2+r\sum\limits_{i=2}^Mn_i+\frac{r+1}{2}\]and $\lambda_i\in\R\setminus\{0\}$ the polynomial $\varphi(y_1,y_2)=y_1^{\nu_1}y_2^{\nu_2}\prod\limits_{i=1}^{M}(y_2-\lambda_iy_1^r)^{n_i}$ satisfies the assumption $\di+1>N\geq \di+1/2$.

We present here a lemma that characterises the mixed homogeneous polynomials with $\di\in[1,2)$ and $s=1$. This will also be used later in Section \ref{s=1}.

\begin{lemma} \label{d<2}
Let $\varphi$ be given by \eqref{factor} with $s=1$ and $1\le d_h(\varphi)<2$. Then one of the following three holds:
\begin{enumerate}
\item[(a)] $n=1$, $\nu_2=0$ and $1\le\nu_1 \le r+1$,
\item[(b)] $n=1$, $\nu_2=1$ and $\nu_1\in\{0,1\}$,
\item[(c)] $n=2$, $\nu_2=0$ and $\nu_1\in\{0,1\}$.
\end{enumerate}
\end{lemma}
\begin{proof}
We have $2(r+1)>d_h(r+1)=\nu_1+r\nu_2+rn$ which implies that $\nu_2+n < 2+2/r \le 3$, i.e. $\nu_2+n\le 2$. This leaves us with three choices for the pair $n$, $\nu_2$, since $n=0$ would mean that $\varphi$ is a monomial.
\begin{enumerate}
\item[(a)] If $n=1$ and $\nu_2=0$, then $2r+2> \nu_1 +r \ge r+1$ and thus $1\le\nu_1 \le r+1$.
\item[(b)] If $n=1$ and $\nu_2=1$, then $2r+2> \nu_1 +2r \ge r+1$ and thus $\nu_1\in\{0,1\}$.
\item[(c)] If $n=2$ and $\nu_2=0$, then $\nu_1\in\{0,1\}$ as above.
\end{enumerate}
\end{proof}

The only polynomials that are under consideration in this section and have $\di\in[1,2)$ are contained in part (c) of Lemma \ref{d<2} for $\nu_1=0$. This is because, first of all, $d_h<2$ implies $3>N\geq 3/2$ and therefore $N=2=n$. Then, if $\nu_1=1$, $d_h=\frac{2r+1}{r+1}\ge \frac53$ and so do not satisy the inequality $N\ge d_h+1/2$. This leaves us with $n=2$, $\nu_1=\nu_2=0$. In that case
\[
d_h + \frac12 = \frac{2r}{r+1} + \frac12 \le 2 \Longleftrightarrow r \in \{2,3\}.
\]

Simple computations show that if $\varphi(y_1,y_2)=(y_2-\lambda y_1^r)^2$, then
\[w(y)=-4\lambda r(r-1)y_2y_1^{r-2}+4\lambda^2r(r-1)y_1^{2r-2}.\]
This implies that $w(y)=-8\lambda(y_2-\lambda y_1^2)$ for $r=2$ and $w(y)=-24\lambda y_1(y_2-\lambda y_1^3)$ for $r=3$.

The $L^p-L^q$ estimates on the line segment
\[\left(\left(\frac{2N}{2N+1},\frac{2N-2}{2N+1}\right),\left(1,1\right)\right]\] (i.e. all points on the line segment joining $(\frac{2N}{2N+1},\frac{2N-2}{2N+1})$ to $(1,1)$, including $(1,1)$, but not including $(\frac{2N}{2N+1},\frac{2N-2}{2N+1})$) follow by the argument of Section \ref{dimline}. Therefore we are only left with establishing estimates on the line segment
\begin{equation}\label{lin1}\left(\left(\frac{2\di+1-N}{\di+1},\frac{2\di-N}{\di+1}\right),\left(1,1\right)\right],\end{equation} 
since the lines $\frac1q = \frac1p - \frac{1}{d_h +1}$ and $\frac{1}{q} = \frac{N+2}{N+1}\frac{1}{p}-\frac{2}{N+1}$ intersect at the point \linebreak$(\frac{2\di+1-N}{\di+1}, \frac{2\di-N}{\di+1})$. The proof of the full region described in part (a) of Theorem \ref{main} then follows by duality.

The proof is again based on an appropriate decomposition of the domain of integration. We focus on the analysis of the case, where $w$ vanishes along the curve \[\mathcal{C}_\lambda=\{(y_1,\lambda y_1^r)\in\R^2: y_1\in\R\},\quad \lambda=\lambda_l\in\R\setminus\{0\},\quad \text{with}\ n_l=N,
 \ l\in\{1,\ldots,M\}.\] We know that
$w$ vanishes to the order $T=2N-3$ along the curve $\mathcal{C}_\lambda$.
We need to show that the operator 
\[\scriptA_{\mathcal{C}_\lambda}f(x)=\int_{\R^2}f(x-\Phi(y))\psi(y)\eta\Big(\frac{y_2-\lambda y_1^r}{\varepsilon y_1^r}\Big)dy\]is bounded on the line
\eqref{lin1} provided $\varepsilon>0$ is chosen to be sufficiently small. That this is sufficient to conclude that $\scriptA$ is also bounded on this line is seen as in the first case, namely $\scriptA=\scriptA_{\mathcal{C}_\lambda}+\scriptA_{{\mathcal{C}_\lambda}^{c}}$, where  
\[\scriptA_{{\mathcal{C}_\lambda}^{c}}f(x)=\int_{\R^2}f(x-\Phi(y))\psi(y)\left(1-\eta\Big(\frac{y_2-\lambda y_1^r}{\varepsilon y_1^r}\Big)\right)dy.\]
The operator $\scriptA_{{\mathcal{C}_\lambda}^{c}}$ is even bounded on the trapezium given by the lines $\frac{1}{p}\leq\frac{1}{q}$, $\frac{1}{q}\geq\frac3p-2$, $\frac1q \ge \frac{1}{3p}$ and $\frac{1}{q} >\frac1p-\frac{1}{\di+1}$, which is a larger region, since for any $\delta>0$, we have
\[\int_{\R^2}\left| w(y)\right|^{-\frac{1}{2\di-2}+\delta}\psi(y)\left(1-\eta\Big(\frac{y_2-\lambda y_1^r}{\varepsilon y_1^r}\Big)\right)dy<\infty.\]
Thus we only need to focus our analysis on $\scriptA_{\mathcal{C}_\lambda}$. A change of variables gives
\[\scriptA_{\mathcal{C}_\lambda}f(x)=\int_{\R^2}f(x-\Phi(y_1,y_2+\lambda y_1^r))\psi(y_1,y_2+\lambda y_1^r)\eta\Big(\frac{y_2}{\varepsilon y_1^r}\Big)dy.\]
Observe that 
\[\Phi(y_1,y_2+\lambda y_1^r)=\left(y_1,y_2+\lambda y_1^r,y_1^{\nu_1}y_2^{N}\left(y_2+\lambda y_1^r\right)^{\nu_2}\prod\limits_{\genfrac{}{}{0pt}{}{i=1}{i\neq l}}^M(y_2-\tilde{\lambda}_iy_1^r)^{n_i}\right),\]where $\tilde{\lambda}_i=\lambda_i-\lambda\in\C\setminus\{0\}$.
Then decomposing bidyadically we have, for some integer $L=L(\lambda,p)$, 
\begin{equation*}\renewcommand{\arraystretch}{1.8}\begin{array}{rcl}
\scriptA_{\mathcal{C}_\lambda}f(x)&\leq&\Sum\limits_{k=L}^{\infty}\Sum\limits_{j=0}^\infty\int_{\R^2}f(x-\Phi(y_1,y_2+\lambda y_1^r))\eta\Big(\frac{y_2}{\varepsilon y_1^r}\Big)\chi_j(y_1)\chi_k(y_2)dy\\
&\leq&\Sum\limits_{k=L}^{\infty}\Sum\limits_{ j\ll\frac{k}{r}}2^{-j-k}\int_{\R^2}f(x-\Phi(2^{-j}y_1,2^{-k}y_2+\lambda 2^{-jr}y_1^r))\chi\otimes\chi(y)dy.
\end{array}\end{equation*}
We also have
\[\int_{\R^2}f(x-\Phi(2^{-j}y_1,2^{-k}y_2+\lambda 2^{-jr}y_1^r))\chi\otimes\chi(y)dy=
\underbrace{\scriptA_{j,k}(f\circ D_{j,k})(D_{j,k}^{-1}x)}_{\tilde{\scriptA}_{j,k}f(x)},\]
where 
\[\scriptA_{j,k}f(x)=\int_{\R^2}f(x-(y_1,\delta_{j,k}y_2+\lambda y_1^r,\varphi_{j,k}(y)))\chi\otimes\chi(y)dy =: f * \mu_{j,k}(x),\]
\[\varphi_{j,k}(y)=y_1^{\nu_1}y_2^{N}(\delta_{j,k}y_2+\lambda y_1^r)^{\nu_2}\prod\limits_{\genfrac{}{}{0pt}{}{i=1}{i\neq l}}^M(\delta_{j,k}y_2-\tilde{\lambda}_iy_1^r)^{n_i},\quad \delta_{j,k}=2^{-k+jr}\ll1,\]
\[D_{j,k}(z_1,z_2,z_3)=(2^{-j}z_1,2^{-jr}z_2,2^{-j\nu_1-jr\nu_2-jr(n-n_l)-kn_l}z_3).\]
In order to see the boundedness along the line \eqref{lin1} first recall that $n_l=N$. Since $\min\{r,N\}\geq2$, taking second derivatives in the $y_1$ and the $y_2$ variables and using van der Corput estimates for oscillatory integrals as in the case $N\ge d_h+1$ (similar oscillatory integrals have also been considered in \cite{zimmermann}), we have $|\hat{\mu}_{j,k}(\xi)|\lsm (1+|\xi|)^{-1/2}$ which implies
\[
\|\scriptA_{j,k}\|_{L^{\frac32}(\R^3)\to L^{3}(\R^3)}\lesssim1.
\] 
Using the weighted strong type $L^{4/3}\to L^4$ estimate of Gressman \cite{GressJGA} for the hypersurface given by
\[
\tilde{\Phi} = \left(\begin{array}{ccc} 1 & 0 & 0 \\ 0 & 1 & 0 \\ 0 & 0 & \delta_{j,k}^{-N} \end{array}\right) \Phi
\]
and using the reparametrisation invariance of the Gaussian curvature, we have
\[\|\scriptA_{j,k}\|_{L^{\frac43}(\R^3)\to L^{4}(\R^3)}\lesssim\delta_{j,k}^{-\frac{1}{4}}.\] Interpolation then gives
\[\|\scriptA_{j,k}\|_{L^{\frac{2(N-\di)+1}{N-\di+1}}(\R^3)\to L^{\frac{2(N-\di)+1}{N-\di}}(\R^3)}\lesssim\delta_{j,k}^{\frac{N-\di-1}{2(N-\di)+1}}.\]
\begin{remark}
We observe that indeed we have
\[\frac23<\frac{N-\di+1}{2(N-\di)+1}\leq\frac34,\]
since it is equivalent to the assumption $\di+1/2\leq N<\di+1$.
\end{remark}
By Lemma \ref{childrenlemma} we have
\begin{align*}
\| \tilde{\scriptA}_{j,k} & \|_{L^{\frac{2(N-\di)+1}{N-\di+1}}(\R^3)\to L^{\frac{2(N-\di)+1}{N-\di}}(\R^3)} \\
& \ \lsm \ \delta_{j,k}^{\frac{N-\di-1}{2(N-\di)+1}} (\det D_{j,k})^{-\frac{1}{2(N-\di)+1}}\\
& \ = \ \delta_{j,k}^{\frac{N-\di-1}{2(N-\di)+1}}2^{\frac{j(r+1)+jr\nu_2+j\nu_1+jr(n-N)+kN}{2(N-\di)+1}}\\
& \ = \ \delta_{j,k}^{\frac{N-\di-1}{2(N-\di)+1}}2^{\frac{j((r+1)(\di+1)-rN)+kN}{2(N-\di)+1}}\\
& \ = \ 2^{-k\frac{N-\di-1}{2(N-\di)+1}}2^{jr\frac{N-\di-1}{2(N-\di)+1}}2^{\frac{j((r+1)(\di+1)-rN)+kN}{2(N-\di)+1}}.
\end{align*}
Since, \[\|\tilde{\scriptA}_{j,k}\|_{L^{1}(\R^3)\to L^{1}(\R^3)}\lesssim1,\] 
it suffices to see that the sum
\[
\Sum\limits_{k=0}^{\infty}\Sum\limits_{ 0\leq j \leq\frac{k}{r}} 2^{-k\left[1-\theta\frac{\di+1}{2(N-\di)+1}\right]} 2^{-jr\left[1-\theta\frac{\di+1}{2(N-\di)+1}\right]},
\]
converges for any $\theta<\frac{2(N-\di)+1}{\di+1}$, which is clear. 
\begin{remark} We observe that indeed $\frac{2(N-\di)+1}{\di+1}\in[0,1]$. Because of the assumption $N\geq\di+1/2$ we have clearly $\frac{2(N-\di)+1}{\di+1}\geq0$.\\On the other hand, \[\frac{2(N-\di)+1}{\di+1}\leq1\Longleftrightarrow N\leq \frac{3\di}{2}.\]
That $N\leq \frac{3\di}{2}$ is true, if $\di+ 1/2\leq N<\di+1$ is seen as follows. First, observe that $\di\geq2$ implies $\frac{3\di}{2}\geq\di+1$, and therefore $N\leq \frac{3\di}{2}$. On the other hand, if $\di<2$, then by the discussion directly after Lemma \ref{d<2}, we have $d_h\ge 4/3$. Thus $3\di/2\ge 2=N$, giving $\frac{2(N-\di)+1}{\di+1}\in[0,1]$.\end{remark}

\section{$\max\{\nu_1,\nu_2\}\ge d_h(\varphi)$}
\label{nu}
In this case we have $\h=\max\{\nu_1,\nu_2\}$ and in this section we will drop the assumption that $\kappa_1<\kappa_2$. This allows us to assume without loss of generality that $\nu_1 \ge \nu_2$. 
\begin{lemma}\label{computationsforw}
Suppose that $\varphi(y)=y_1^nQ(y_1,y_2)$, $n\geq1$, where \[Q(y_1,y_2)=cy_2^m+\mathcal{O}(y_1),\quad c\neq0,\ m\geq1.\]
Then $w(y)=y_1^{2n-2}\tilde{Q}(y)$, where
 \[\tilde{Q}(y)=c^2nm(1-n-m)y_2^{2m-2}+\mathcal{O}(y_1).\]
\end{lemma}
\begin{proof}
Case $n=1$ is easy to see, so we may and shall assume that $n\geq2$. Simple computations show
\[\partial_1^2\varphi(y_1,y_2)=y_1^{n-2}\left(n(n-1)Q(y)+2ny_1\partial_1Q(y)+y_1^2\partial_1^2Q(y)\right),\]
\[\partial_1\partial_2\varphi(y_1,y_2)=y_1^{n-1}\left(n\partial_2Q(y)+y_1\partial_1\partial_2Q(y)\right),\]
\[\partial_2^2\varphi(y_1,y_2)=y_1^n\partial_2^2Q(y),\]
and
\begin{equation*}\renewcommand{\arraystretch}{1.8}\begin{array}{rcl}
w(y)&=&\partial_1^2\varphi(y)\partial_2^2\varphi(y)-(\partial_1\partial_2\varphi(y))^2\\
&=&y_1^{2n-2}\left(n(n-1)Q(y)\partial_2^2Q(y)+\mathcal{O}(y_1)\right)-y_1^{2n-2}\left(n\partial_2Q(y)+\mathcal{O}(y_1)\right)^2\\
&=&y_1^{2n-2}\left(n(n-1)Q(y)\partial_2^2Q(y)-n^2(\partial_2Q(y))^2+\mathcal{O}(y_1)\right).
\end{array}\end{equation*}
Furthermore, we have
\[n(n-1)Q(y)\partial_2^2Q(y)-n^2(\partial_2Q(y))^2=c^2mn(1-n-m)y_2^{2m-2}+\mathcal{O}(y_1),\]
which gives $w(y)=y_1^{2n-2}\tilde{Q}(y)$.
\end{proof}
By assumption we have $\varphi(y)=y_1^{\nu_1}\left(y_2^{\nu_2+sn}+\mathcal{O}(y_1)\right)$ and $\nu_1\geq\max\{1,\di\}$. Lemma \ref{computationsforw} gives 
\[
w(y)=y_1^{2\nu_1-2}\big[\nu_1(\nu_2+sn)(1-\nu_1 - \nu_2 - sn)y_2^{2\nu_2+2sn-2} + \mathcal{O}(y_1)\big].
\]
Obviously $\nu_1(\nu_2+sn)(1-\nu_1 - \nu_2 - sn)\neq0$, which gives that $w$ has a zero of order $2\nu_1-2$ along the axis $y_1=0$. If $w$ is not a monomial, then $d_h(w)=2d_h(\varphi)-2 \le 2\nu_1 - 2$. If $w$ happens to be a monomial, say $w(y)=cy_1^{2\nu_1-2}y_2^B$, then by considering the mixed homogeneity of derivatives of mixed homogeneous polynomials, we should have
\[
\frac{(2\nu_1-2)\kappa_1+B\kappa_2}{2(1-\kappa_1-\kappa_2)} = 1,
\]
which is equivalent to $\nu_1\kappa_1+(1+B/2)\kappa_2=1$. Since we have also assumed that $\nu_1\ge d_h(\varphi) \Leftrightarrow \nu_1\kappa_1 \ge 1 - \nu_1\kappa_2$, we have $1 - (1+B/2)\kappa_2 \ge 1 - \nu_1\kappa_2$ which is equivalent to $B\le 2\nu_1-2$. Hence, whether $w$ is a monomial or not, $h(w)= 2\nu_1 - 2$. By the argument of Section \ref{dimline}, we obtain the boundedness of $\scriptA$ on the line segment 
\[\left(\left(\frac{2\nu_1+1}{2\nu_1+2},\frac{2\nu_1-1}{2\nu_1+2}\right),(1,1)\right].\]
The point $\left(\frac{2\nu_1+1}{2\nu_1+2},\frac{2\nu_1-1}{2\nu_1+2}\right)$ is the intersection of the lines $\frac1q = \frac1p - \frac{1}{\nu_1+1}$ and $\frac1q = \frac3p -2$. The boundedness on the rest of the region described in part (b) of Theorem \ref{main}, follows by interpolation.

\section{$N<d_h(\varphi) + 1/2$, $\max\{\nu_1,\nu_2\} < d_h(\varphi)$, $s\ge 2$}
\label{s>1}
We split our analysis in several subcases starting from the simpler ones. Here, by Corollary \ref{hompolynomialcor}, we have $n<d_h(\varphi)$. We will require the following lemma.
\begin{lemma}\label{computationsforw2}
Suppose that $\varphi(y_1,y_2)=y_2^M+y_1^AQ(y_1,y_2)$, $A, M\in\N$, $\min\{A,M\}\geq2$ and 
$Q(y_1,y_2)=cy_2^B+\mathcal{O}(y_1)$, $c\neq0$, $B\in\N_0$. Then
\[w(y)=y_1^{A-2}\left(cA(A-1)M(M-1)y_2^{B+M-2}+\mathcal{O}(y_1)\right).\]
In particular, $w$ vanishes to the order $A-2$ transversally to the line $y_1=0$.
\end{lemma}
\begin{proof}
Computations show
\[\partial_1^2\varphi(y_1,y_2)=A(A-1)y_1^{A-2}Q(y)+2Ay_1^{A-1}\partial_1Q(y)+y_1^{A}\partial_1^2Q(y),\]
\[\partial_1\partial_2\varphi(y_1,y_2)=Ay_1^{A-1}\partial_2Q(y)+y_1^A\partial_1\partial_2Q(y),\]
\[\partial_2^2\varphi(y_1,y_2)=M(M-1)y_2^{M-2}+y_1^A\partial_2^2Q(y).\]This gives
\begin{equation*}\renewcommand{\arraystretch}{1.8}\begin{array}{rcl}
w(y)&=&\partial_1^2\varphi(y)\partial_2^2\varphi(y)-(\partial_1\partial_2\varphi(y))^2\\
&=&y_1^{A-2}\Big(A(A-1)Q(y)+\mathcal{O}(y_1)\Big)\Big(M(M-1)y_2^{M-2}+\mathcal{O}(y_1)\Big)\\
&& - \ y_1^{2A-2}\Big(A\partial_2Q(y)+\mathcal{O}(y_1)\Big)^2.\\
\end{array}\end{equation*}
Clearly, $2A-2>A-2$ and
\begin{equation*}\renewcommand{\arraystretch}{1.8}\begin{array}{rcl}
&&\Big(A(A-1)Q(y)+\mathcal{O}(y_1)\Big)\Big(M(M-1)y_2^{M-2}+\mathcal{O}(y_1)\Big)\\
&=&A(A-1)M(M-1)Q(y)y_2^{M-2}+\mathcal{O}(y_1)\\
&=&cA(A-1)M(M-1)y_2^{B+M-2}+\mathcal{O}(y_1).\end{array}\end{equation*}
We conclude that
\[w(y)=y_1^{A-2}\left(cA(A-1)M(M-1)y_2^{B+M-2}+\mathcal{O}(y_1)\right).\]
Since $\min\{A,M\}\geq2$, we have $cA(A-1)M(M-1)\neq0$ and this shows that $w$ vanishes to the order $A-2$ transversally to the line $y_1=0$.
\end{proof}



Before continuing with the analysis in this section, we note a corollary of Lemmas \ref{computationsforw} and \ref{computationsforw2} which was mentioned in the introduction.

\begin{corollary} \label{flat}
Let $\varphi$ be a mixed homogeneous polynomial. Then $w\not\equiv 0$.
\end{corollary}

\begin{proof}
We only need to consider the case $\nu_1=\nu_2=0$ otherwise Lemma \ref{computationsforw} would apply and give us a form for $w$ which clearly cannot satisfy $w\equiv 0$. Given that $\nu_1=\nu_2=0$, and because $\varphi$ being mixed homogeneous implies that $A\ge 2$ in Lemma \ref{computationsforw2}, this lemma applies and gives a form for $w$ which again cannot satisfy $w\equiv 0$.
\end{proof}

\subsection{Case $\min\{\nu_1,\nu_2\}\geq1$}
Here we first observe that Lemma \ref{computationsforw} shows that $w$ vanishes to the order $2\nu_i-2$ transversally to the line $y_i=0$, $i=1,2$. If $w$ were a monomial, then $w(y)=cy_1^{2\nu_1-2}y_2^{2\nu_2-2}$ and
\[
\frac{(2\nu_1-2)\kappa_1}{2(1-\kappa_1-\kappa_2)} + \frac{(2\nu_2-2)\kappa_2}{2(1-\kappa_1-\kappa_2)} = \frac{\nu_1\kappa_1+\nu_2\kappa_2-\kappa_1-\kappa_2}{1-\kappa_1-\kappa_2} =1,
\]
implying $\nu_1\kappa_1+\nu_2\kappa_2 =1$. However that contradicts the assumption $\max\{\nu_1,\nu_2\}<d_h(\varphi)$, so $w$ cannot be a monomial. In addition, $\max\{\nu_1,\nu_2\}<d_h(\varphi)=\h$, implies \[\max\{2\nu_1-2,2\nu_2-2\}\leq2\di-2=d_h(w).\]
Since $s\ge 2$, we know that the multiplicity of any other root of $w$ which does not lie on a coordinate axis is bounded by the homogeneous distance of $w$, which is equal to $2d_h(\varphi)-2$. We conclude that the multiplicity of any root of $w$ is bounded by its homogeneous distance. This clearly shows that $\scriptA$ is bounded on the trapezium given by $\frac1q\le\frac1p$, $\frac1q\ge\frac3p-2$, $\frac1q\ge\frac{1}{3p}$ and $\frac1q >\frac1p-\frac{1}{\di+1}$.

\subsection{Case $\min\{\nu_1,\nu_2\}=0$} \label{monarg}
We split this case into two subcases. We first discuss the simpler subcase.
\subsubsection{Subcase $\nu_1\geq1$, $\nu_2=0$}
In this subcase we have $(\nu_1+rn,0)\in\mathcal{T}(\varphi)$. Since the Taylor support $\mathcal{T}(\varphi)$ consists of at least two points, we find a point $(A,B)\in\mathcal{T}(\varphi)$ with $1\leq B\leq\beta$ for all $(\alpha,\beta)\in\mathcal{T}(\varphi)\setminus\{(\nu_1+rn,0)\}$. Since $r/s \notin \N$, we have $B\geq2$. Lemma \ref{computationsforw2} then shows that $w$ vanishes to the order $B-2$ transversally to the line $y_2=0$. We will need to show that $B-2\leq 2\di-2$, i.e. $B\leq2\di$. This is seen as follows. Observe first that clearly $B\leq sn$ and $2\di=2\frac{s\nu_1+rsn}{r+s}$. Furthermore,
\[2\di\geq sn\Longleftrightarrow 2s\nu_1+2rsn\geq rsn+s^2n\Longleftrightarrow 2\nu_1+rn\geq sn,\]
and the last statement is true, since $r>s$. Hence $2\di\geq sn\geq B$. Here $w$ is not a monomial because, if it were, then
\[
\frac{(2\nu_1-2)\kappa_1}{2(1-\kappa_1-\kappa_2)} + \frac{(B-2)\kappa_2}{2(1-\kappa_1-\kappa_2)} = 1,
\]
implying $1 = \nu_1\kappa_1+B\kappa_2 /2 < (\kappa_1+\kappa_2)/d_h = 1$. Therefore $h(w)=2d_h-2$ and $\scriptA$ is bounded on the trapezium given by $\frac1q\le\frac1p$, $\frac1q\ge\frac3p-2$, $\frac1q\ge\frac{1}{3p}$ and $\frac1q >\frac1p-\frac{1}{\di+1}$.

\subsubsection{Subcase $\nu_1=0$, $\nu_2\geq0$}
Then $(0,\nu_2+sn)\in\mathcal{T}(\varphi)$ and there is a point $(A,B)\in\mathcal{T}(\varphi)$ with $1\leq A\leq\alpha$ for any $(\alpha,\beta)\in\mathcal{T}(\varphi)\setminus\{(0,\nu_2+sn)\}$. Since $r>s$ we have $A\geq2$. We conclude using Lemma \ref{computationsforw2} that $w$ vanishes to the order $A-2$ transversally to the line $y_1=0$. The method of proof resembles the proof in Sections \ref{N>d+1} and \ref{inter}, but we will need the following lemma that characterises polynomials under consideration here with $d_h<2$. 

\begin{lemma}\label{distless2}Let $\varphi$ be given by \eqref{factor} with $s\ge 2$, $\nu_1=0$ and $1\le d_h(\varphi)<2$. Then one of the following three holds:
\begin{enumerate}
\item[(a)] $n=1$, $\nu_2=0$, $s=2$ and $r$ is odd,
\item[(b)] $n=1$, $\nu_2=0$, $s=3$ and $r\in\{4,5\}$,
\item[(c)] $n=1$, $\nu_2=1$, $s=2$ and $r=3$.
\end{enumerate}
\end{lemma}
\begin{proof}First observe that we have
\[2>\di=\frac{r\nu_2+rsn}{r+s}\geq\frac{rsn}{r+s}\geq\frac{n}{\frac{1}{r}+\frac1s}\geq\frac{n}{\frac12+\frac13}=\frac65n.\]
This gives $n=1$, which in turn gives \[\di=\frac{\nu_2r+rs}{r+s}.\]
We have $\nu_2\in\{0,1\}$, since if $\nu_2\geq2$, we have a contradiction by
\[2>\di=\frac{r\nu_2+rs}{r+s}\geq\frac{2r+rs}{r+s}\geq\frac{2s+2r}{r+s}=2.\]
If $\nu_2=0$, then
\[
2>\frac{rs}{r+s}=\frac{1}{\frac1s+\frac1r} \Rightarrow \frac4s > 1 \Rightarrow s\le 3.
\]
If $s=2$, $r$ has to be odd since $\gcd(s,r)=1$. If $s=3$, then 
\[
2>\frac{3r}{r+3} \Rightarrow s<r\le 5.
\]
If $\nu_2=1$, then
\[
2>\frac{r+rs}{r+s} \Rightarrow \frac1s + \frac2r >1 \Rightarrow s\le 2 \Rightarrow s<r\le 3.
\]
\end{proof}

\paragraph{$A\leq2\di$} In this case we observe, as in the previous subcase, that $A\leq2\di$ implies that $w$ is not a monomial and $\scriptA$ is bounded on the trapezium given by $\frac1q\le\frac1p$, $\frac1q\ge\frac3p-2$, $\frac1q\ge\frac{1}{3p}$ and $\frac1q >\frac1p-\frac{1}{\di+1}$.
\vspace{1em}
\paragraph{$A>2\di$} \label{A>2d} Here the region is then determined by $\frac1q \le\frac1p$, $\frac1q \ge\frac3p-2$, $\frac1q\ge\frac{1}{3p}$, $\frac1q >\frac1p-\frac{1}{\di+1}$, $\frac1q>\frac{2A+1}{A+1}\frac1p-1$ and $\frac{1}{q}>\frac{A+1}{2A+1}\frac{1}{p}-\frac{1}{2A+1}$. This is contained in part (c) of Theorem \ref{main} for $T=A-2$ (cf. Lemma \ref{computationsforw2}). We need to show the boundedness of the operator $\mathcal{A}$ along the line segment 
\begin{equation}\label{Aline}\left(\left(\frac{(A+1)\di}{A(\di+1)},\frac{\di(A+1)-A}{A(\di+1)}\right),(1,1)\right],\end{equation}
the point $\left(\frac{(A+1)\di}{A(\di+1)},\frac{\di(A+1)-A}{A(\di+1)}\right)$ being the intersection of the lines $\frac1q = \frac1p - \frac{1}{d_h+1}$ and $\frac1q = \frac{2A+1}{(A+1)p} -1$.
The rest of the proof then follows by the argument of Section \ref{dimline} and interpolation. Observe that line \eqref{Aline} is parametrised by $\frac1q=\frac{2A-\di}{A-\di}\frac1p-\frac{A}{A-\di}$ and intersects the off-diagonal at the point
\[\left(\frac{2A-\di}{3A-2\di},\frac{A-\di}{3A-2\di}\right).\]
It suffices to show that for $\varepsilon>0$ sufficiently small the operator  \[\mathcal{A}_{\mathcal{C}_1}f(x)=\int_{\R^2}f(x-\Phi(y))\psi(y)\eta\left(\frac{y_1^r}{\varepsilon y_2^{s}}\right)dy\]is bounded as an operator from $L^p(\R^3)$ to $L^q(\R^3)$ for every $(\frac{1}{p},\frac{1}{q})$ on the line in \eqref{Aline}.
The ``remainder part'' $\mathcal{A}-\mathcal{A}_{\mathcal{C}_1}$ of the operator $\mathcal{A}$ is bounded on the larger trapezium given by $\frac1q \le\frac1p$, $\frac1q \ge\frac3p-2$, $\frac1q\ge\frac{1}{3p}$ and $\frac1q >\frac1p-\frac{1}{\di+1}$.
Bidyadic decomposition and change of variables gives 
\begin{equation*}\renewcommand{\arraystretch}{1.8}\begin{array}{rcl}
\mathcal{A}_{\mathcal{C}_1}f(x)&\leq&\Sum\limits_{j=0}^{\infty}\Sum\limits_{k=0}^\infty\int_{\R^2}f(x-\Phi(y))\eta\left(\frac{y_1^r}{\varepsilon y_2^s}\right)\chi_j(y_1)\chi_k(y_2)dy\\
&\leq&\Sum\limits_{j=0}^{\infty}\Sum\limits_{ k\ll\frac{jr}{s}}2^{-j-k}\int_{\R^2}f(x-\Phi(2^{-j}y_1,2^{-k}y_2))\chi\otimes\chi(y)dy\\
&=&\Sum\limits_{j=0}^{\infty}\Sum\limits_{  k\ll\frac{jr}{s}}2^{-j-k}\underbrace{\mathcal{A}_{j,k}(f\circ D_{j,k})(D_{j,k}^{-1}x)}_{\tilde{\mathcal{A}}_{j,k}f(x)},\end{array}\end{equation*}
where \[\mathcal{A}_{j,k}f(x)=\int_{\R^2}f(x-(y,\varphi_{j,k}(y)))\chi\otimes\chi(y)dy,\]
\[\varphi_{j,k}(y)=y_2^{\nu_2}\prod\limits_{i=1}^M(y_2^s-\lambda_i(\delta_{j,k}y_1)^r)^{n_i},\quad \delta_{j,k}=2^{\frac{ks}{r}-j},\]
\[D_{j,k}(z_1,z_2,z_3)=(2^{-j}z_1,2^{-k}z_2,2^{-k\nu_2-ksn}z_3).\]
Observe that $\varphi_{j,k}(y_1,y_2)=\varphi(\delta_{j,k}y_1,y_2)$ and that $\delta_{j,k}\ll1$. This gives \[\det\varphi_{j,k}^{\prime\prime}(y_1,y_2)=\delta_{j,k}^2w(\delta_{j,k}y_1,y_2)\sim\delta_{j,k}^{2+A-2}=\delta_{j,k}^{A}\quad\text{on}\ \supp\chi\otimes\chi.\]
This implies $\|\mathcal{A}_{j,k}\|_{L^{\frac{4}{3}}(\R^3)\to L^{4}(\R^3)}\lesssim \delta_{j,k}^{-\frac{A}{4}}$. Because $\nu_2+sn\geq2$, taking two derivatives in the $y_2$ variable we obtain, using van der Corput estimates in the same manner as in Sections \ref{N>d+1} and \ref{inter}, that
\[\|\mathcal{A}_{j,k}\|_{L^{\frac{3}{2}}(\R^3)\to L^{3}(\R^3)}\lesssim 1.\] Interpolation then gives
\[\|\mathcal{A}_{j,k}\|_{L^{\frac{3A-2\di}{2A-\di}}(\R^3)\to L^{\frac{3A-2\di}{A-\di}}(\R^3)}\lesssim \left(\delta_{j,k}^{-\frac{A}{4}}\right)^{\frac{4\di}{3A-2\di}}=\delta_{j,k}^{-\frac{A\di}{3A-2\di}}.\]
\begin{remark}We observe that indeed the point $\left(\frac{2A-\di}{3A-2\di},\frac{A-\di}{3A-2\di}\right)$ lies on the open line segment joining the points $\left(\frac23,\frac13\right)$ and $\left(\frac34,\frac14\right)$, since 
\[\frac23<\frac{2A-\di}{3A-2\di}<\frac34\Longleftrightarrow 0<\di<\frac{A}{2}. \]
\end{remark}
Applying Lemma \ref{childrenlemma}, we have
\begin{equation*}\renewcommand{\arraystretch}{1.8}\begin{array}{rcl}
\|\tilde{\mathcal{A}}_{j,k}\|_{L^{\frac{3A-2\di}{2A-\di}}(\R^3)\to L^{\frac{3A-2\di}{A-\di}}(\R^3)}&\lesssim&\delta_{j,k}^{-\frac{A\di}{3A-2\di}}\left|\det D_{j,k}\right|^{-\frac{A}{3A-2\di}}\\
&=&\delta_{j,k}^{-\frac{A\di}{3A-2\di}}(2^{j+k+k(\nu_2+sn)})^{\frac{A}{3A-2\di}}\\
&=&2^{j\frac{A\di+A}{3A-2\di}}2^{\frac{k}{3A-2\di}(A+A(\nu_2+sn)-A\di\frac{s}{r})}\\
&=&2^{j\frac{A(\di+1)}{3A-2\di}}2^{\frac{kA}{3A-2\di}(1+\frac{r+s}{r}\di-\di\frac{ s}{r})}\\
&=&2^{j\frac{A(\di+1)}{3A-2\di}}2^{k\frac{A(\di+1)}{3A-2\di}}.
\end{array}\end{equation*}
Clearly $\|\tilde{\mathcal{A}}_{j,k}\|_{L^{1}(\R^3)\to L^{1}(\R^3)}\lesssim1$. Simple computations show that for $\theta=\frac{3A-2\di}{A(\di+1)}$ we have

\[1-\theta+\theta\frac{2A-\di}{3A-2\di}=\frac{A+1}{A}\frac{\di}{\di+1}\]
and
\[1-\theta+\theta\frac{A-\di}{3A-2\di}=\frac{\di(A+1)-A}{A(\di+1)}.\]
Therefore, we only need to show that 
\[\Sum\limits_{j=0}^{\infty}2^{-j\left[1-A\theta\frac{(\di+1)}{3A-2\di}\right]}\Sum\limits_{0\leq k\leq\frac{jr}{s}}2^{-k\left[1-A\theta\frac{(1+\di)}{3A-2\di}\right]}\]converges for every $\theta<\frac{3A-2\di}{A(\di+1)}$, which is obvious.
\begin{remark}We notice that $\frac{3A-2\di}{A(\di+1)}\in[0,1]$. Because of the assumption $A>2\di$, we clearly have $\frac{3A-2\di}{A(\di+1)}>0$.
Furthermore,
\[\frac{3A-2\di}{A(\di+1)}\leq1\Longleftrightarrow \frac{2A}{A+2}\leq\di,\]
and the last inequality is obviously true if $\di\geq2$. On the other hand, if $\di<2$, we use Lemma \ref{distless2}. Observe that then in all cases $A=r$. In the case where $\varphi(y_1,y_2)=y_2(y_2^2+\lambda y_1^3)$, the assumption $A>2\di$ is not satisfied, because
\[A=r=3<\frac{18}{5}=2\di.\] In the other cases we always have for any $r>s\geq2$
\[\frac{2A}{A+2}=\frac{2r}{r+2}\leq\frac{rs}{r+s}=\di.\]
This gives $\frac{3A-2\di}{A(\di+1)}\leq 1$.
\end{remark}

\section{$N<d_h(\varphi) + 1/2$, $\max\{\nu_1,\nu_2\} < d_h(\varphi)$, $s=1$}
\label{s=1}

We first focus on the case where $w$ vanishes on the $y_1$-axis or the $y_2$-axis to the order $T$. This puts us within part (c) of Theorem \ref{main}. By a similar argument to the one in Section \ref{monarg} we see that if $w$ were a monomial then $T\ge 2d_h-2$. Hence if $T\le 2d_h-2$, then $h(w)=2d_h-2$ and the proof is simply the argument of Section \ref{dimline} and interpolation. Note that for $T\le 2d_h-2$, the line $\frac{1}{q}=\frac{2T+5}{T+3}\frac{1}{p}-1$ for instance is redundant since it lies below the line $\frac1q=\frac1p-\frac{1}{\di+1}$.

For $T>2d_h-2$, clearly $\nu_1=0$ or $\nu_2=0$ if the vanishing is along the $y_2$-axis or the $y_1$-axis, respectively. An example of this is the polynomial $\varphi(y_1,y_2) = y_1^5 +y_2y_1^3+\frac{9}{40} y_2^2y_1$, for which $T=2$ and $d_h=5/3$. The proof is almost identical to the previous section and we therefore omit the details. The only difference is that for the last step we require the corresponding double sum to be convergent for all $0 \le \theta < \frac{3T-2\di+6}{(T+2)(\di+1)}$ and we should check that $\frac{3T-2\di+6}{(T+2)(\di+1)} \le 1$. This is equivalent to 
\begin{equation} \label{check1}
0\le T(d_h-2) +4d_h-4,
\end{equation}
which is easy to see if $d_h \ge 2$. For $d_h<2$, we use Lemma \ref{d<2}. The polynomials that fall under part (a) of Lemma \ref{d<2} have $d_h = \frac{\nu_1+r}{r+1} \le\nu_1$ and therefore are not among those considered in this section. For the polynomials that fall under parts (b) and (c) of Lemma \ref{d<2} with $\nu_1=0$, we have $T=r-2$ and $d_h = \frac{2r}{r+1}$ and \eqref{check1} can be verified directly. Also, a direct computation shows that the polynomials that fall under part (b) with $\nu_1=1$ do not have Hessian determinants that vanish along the axes. Finally, the polynomials that fall under part (c) with $\nu_1=1$ have $T\le 2$ and $d_h\ge 5/3$ and so also satisfy \eqref{check1}.

We next focus on the analysis of the case, where $w$ vanishes along some curve $\mathcal{C}_\lambda=\{(y_1,\lambda y_1^r)\in\R^2: y_1\in\R\}$, $\lambda\in\R\setminus\{0\}$, to some order $T>2\di-2$ (if $T\le 2d_h-2$, the estimates follow by the usual interpolation arguments).
We need to show that the operator 
\[\scriptA_{\mathcal{C}_\lambda}f(x)=\int_{\R^2}f(x-\Phi(y))\psi(y)\eta\Big(\frac{y_2-\lambda y_1^r}{\varepsilon y_1^r}\Big)dy\]is bounded, for sufficiently small $\varepsilon>0$, on the line segments


\[\left(\left(\frac{(T+3)d_h}{(T+2)(d_h+1)},\frac{T(\di-1)+3\di-2}{(T+2)(\di+1)}\right),(1,1)\right],\] if $\lambda=\lambda_l$ for some $l\in\{1,\ldots,M\}$, in which case $n_l=1$, or  
\[\left(\left(\frac{T(2d_h-1)+6d_h}{2(T+4)(d_h+1)},\frac{T(2\di-3)+12\di-8}{2(T+4)(d_h+1)}\right),(1,1)\right],\] if $\lambda\not=\lambda_l$ for any $l\in\{1,\ldots,M\}$.
The point $\left(\frac{(T+3)d_h}{(T+2)(d_h+1)},\frac{T(\di-1)+3\di-2}{(T+2)(\di+1)}\right)$ is the intersection of the lines $\frac1q = \frac{2T+5}{(T+3)p}-1$ and $\frac1q = \frac1p - \frac{1}{d_h(\varphi)+1}$, whereas the point $\left(\frac{T(2d_h-1)+6d_h}{2(T+4)(d_h+1)},\frac{T(2\di-3)+12\di-8}{2(T+4)(d_h+1)}\right)$ is the intersection of the lines $\frac1q = \frac{5}{3p} - \frac{2T+12}{3T+12}$ and $\frac1q = \frac1p - \frac{1}{d_h(\varphi)+1}$.
Changing variables we obtain
\[\scriptA_{\mathcal{C}_\lambda}f(x)=\int_{\R^2}f(x-\Phi(y_1,y_2+\lambda y_1^r))\psi(y_1,y_2+\lambda y_1^r)\eta\Big(\frac{y_2}{\varepsilon y_1^r}\Big)dy.\]
Note that 
\[\Phi(y_1,y_2+\lambda y_1^r)=\left(y_1,y_2+\lambda y_1^r,y_1^{\nu_1}\left(y_2+\lambda y_1^r\right)^{\nu_2}\prod\limits_{i=1}^M(y_2-\tilde{\lambda}_iy_1^r)^{n_i}\right),\]where $\tilde{\lambda}_i=\lambda_i-\lambda\in\C$.
Then, for some integer $L=L(\lambda,r)$, we obtain using a bidyadic decomposition 
\begin{equation*}\renewcommand{\arraystretch}{1.8}\begin{array}{rcl}
\scriptA_{\mathcal{C}_\lambda}f(x)&\leq&\Sum\limits_{k=L}^{\infty}\Sum\limits_{j=0}^\infty\int_{\R^2}f(x-\Phi(y_1,y_2+\lambda y_1^r))\eta\Big(\frac{y_2}{\varepsilon y_1^r}\Big)\chi_j(y_1)\chi_k(y_2)dy\\
&\leq&\Sum\limits_{k=L}^{\infty}\Sum\limits_{ j\ll\frac{k}{r}}2^{-j-k}\int_{\R^2}f(x-\Phi(2^{-j}y_1,2^{-k}y_2+\lambda 2^{-jr}y_1^r))\chi\otimes\chi(y)dy.
\end{array}\end{equation*}
We treat the two cases separately.
 
\vspace{1em}
\noindent\textit{Case 1:} $\tilde{\lambda}_l=0$ for some $l\in\{1,\ldots,M\}$. This case is equivalent to $\lambda=\lambda_l$ for some $l\in\{1,\ldots,M\}$, and clearly $n_l=1$. We then have
\[\int_{\R^2}f(x-\Phi(2^{-j}y_1,2^{-k}y_2+\lambda 2^{-jr}y_1^r))\chi\otimes\chi(y)dy=\underbrace{\scriptA_{j,k}(f\circ D_{j,k})(D_{j,k}^{-1}x)}_{=:\tilde{\scriptA}_{j,k}f(x)},\]
where \[\scriptA_{j,k}f(x)=\int_{\R^2}f(x-(y_1,\delta_{j,k}y_2+\lambda y_1^r,\varphi_{j,k}(y)))\chi\otimes\chi(y)dy,\]
\[\varphi_{j,k}(y)=y_1^{\nu_1}y_2(\delta_{j,k}y_2+\lambda y_1^r)^{\nu_2}\prod\limits_{\genfrac{}{}{0pt}{}{i=1}{i\neq l}}^M(\delta_{j,k}y_2-\tilde{\lambda}_iy_1^r)^{n_i},\quad \delta_{j,k}=2^{-k+jr}\ll1,\]
and
\[D_{j,k}(z_1,z_2,z_3)=(2^{-j}z_1,2^{-jr}z_2,2^{-j\nu_1-jr\nu_2-jr(n-1)-k}z_3).\]
Using van der Corput estimates for oscillatory integrals as in Lemma \ref{decaylemma} (the only difference being the use of a mixed second derivative at the last step instead of a second derivative in the $y_2$ variable) it can be seen that 
\[ \|\scriptA_{j,k}\|_{L^{\frac32}(\R^3)\to L^{3}(\R^3)}\lesssim1.\] 
Using the weighted $L^{4/3}\to L^4$ convolution estimate of Gressman \cite{GressJGA} and arguing as in Section \ref{inter} we have 
\[\|\scriptA_{j,k}\|_{L^{\frac43}(\R^3)\to L^{4}(\R^3)}\lesssim\delta_{j,k}^{-\frac{T+2}{4}}.\] 
Interpolating between the last two estimates gives
\[\|\scriptA_{j,k}\|_{L^{\frac{3T-2\di+6}{2(T+2)-\di}}(\R^3)\to L^{\frac{3T-2\di+6}{T+2-\di}}(\R^3)}\lesssim \left(\delta_{j,k}^{-\frac{T+2}{4}}\right)^{\frac{4\di}{3T-2\di+6}}=\delta_{j,k}^{-\frac{(T+2)\di}{3T-2\di+6}}.\]
Using Lemma \ref{childrenlemma} we have
\begin{eqnarray*}
\|\tilde{\scriptA}_{j,k}\|_{L^{\frac{3T-2\di+6}{2(T+2)-\di}}(\R^3)\to L^{\frac{3T-2\di+6}{T+2-\di}}(\R^3)}
&\lesssim &\delta_{j,k}^{-\frac{(T+2)\di}{3T-2\di+6}} (\det D_{j,k})^{-\frac{T+2}{3T-2\di+6}}\\
&= &\delta_{j,k}^{-\frac{(T+2)\di}{3T-2\di+6}}2^{{\frac{(T+2)(j(r+1)+jr\nu_2+j\nu_1+jr(n-1)+k)}{3T-2\di+6}}} \\
&= &\delta_{j,k}^{-\frac{(T+2)\di}{3T-2\di+6}}2^{\frac{(j(\di r +\di +1)+k)(T+2)}{3T-2\di+6}}\\
&=& 2^{k\frac{(T+2)(\di+1)}{3T-2\di+6}}2^{j\frac{(T+2)(\di+1)}{3T-2\di+6}}. 
\end{eqnarray*}
Clearly $\|\tilde{\scriptA}_{j,k}\|_{L^{1}(\R^3)\to L^{1}(\R^3)}\lesssim1$. In order to have the desired result we observe that 
\[\Sum\limits_{k=0}^{\infty}\Sum\limits_{ 0\leq j \leq\frac{k}{r}} 2^{-k\left[1-\theta \frac{(T+2)(\di+1)}{3T-2\di+6}\right]}2^{-j\left[1-\theta \frac{(T+2)(\di+1)}{3T-2\di+6}\right]}<\infty
\]
for any $0\le \theta<\frac{3T-2\di+6}{(T+2)(\di+1)}\le 1$. Note that the last inequality is the same as \eqref{check1} which again is clear when $d_h \ge 2$. For $d_h<2$, we use Lemma \ref{d<2} and observe that the only part of that lemma that contains polynomials under consideration in this case is part (b) for $\nu_1=1$. In this case, considering the degree of the $y_2$ variable, one can see that $T\le 2$, $d_h \ge 5/3$ and \eqref{check1} can then be verified directly.

\vspace{1em}
\noindent\textit{Case 2:} $\tilde{\lambda}_l\not=0$ for any $l\in\{1,\ldots,M\}$. In this case we have $\lambda\neq\lambda_l$ for all $l\in\{1,\ldots,M\}$. An example where this occurs is $\varphi(y_1,y_2)=y_1(y_2+y_1^3)(y_2+ay_1^3)$, where $a=\frac{5+\sqrt{21}}{2}$.
Then $w(y)=-4(y_2-by_2^3)^2$, where $b=\frac{7+\sqrt{21}}{2}$. In this example $d_h=7/4$ and $2d_h-2<2=T$.

We have
\[\int_{\R^2}f(x-\Phi(2^{-j}y_1,2^{-k}y_2+\lambda 2^{-jr}y_1^r))\chi\otimes\chi(y)dy=\underbrace{\scriptA_{j,k}(f\circ D_{j,k})(D_{j,k}^{-1}x)}_{\tilde{\scriptA}_{j,k}f(x)}\]
where \[\scriptA_{j,k}f(x)=\int_{\R^2}f(x-(y_1,\delta_{j,k}y_2+\lambda y_1^r,\varphi_{j,k}(y)))\chi\otimes\chi(y)dy =: f * \mu_{j,k}(x),\]
\[\varphi_{j,k}(y)=y_1^{\nu_1}(\delta_{j,k}y_2+\lambda y_1^r)^{\nu_2}\prod\limits_{i=1}^M(\delta_{j,k}y_2-\tilde{\lambda}_iy_1^r)^{n_i},\quad \delta_{j,k}=2^{-k+jr}\ll1,\]
\[D_{j,k}(z_1,z_2,z_3)=(2^{-j}z_1,2^{-jr}z_2,2^{-j\nu_1-jr\nu_2-jrn}z_3).\]

Considering the sum of the absolute values of the second and third derivative in the $y_1$ variable and using van der Corput estimates of Bj\"ork type (see Lemma 1.6 in \cite{Domar}) for oscillatory integrals it is easily seen that
\[
|\hat{\mu_{j,k}}(\xi)|\lsm (1+|\xi|)^{-\frac{1}{3}},
\]
which implies that 
\[ \|\scriptA_{j,k}\|_{L^{\frac85}(\R^3)\to L^{\frac83}(\R^3)}\lesssim1.\] 
Using the weighted $L^{4/3}\to L^4$ convolution estimate of Gressman \cite{GressJGA} and arguing as in Section \ref{inter} we have 
\[\|\scriptA_{j,k}\|_{L^{\frac43}(\R^3)\to L^{4}(\R^3)}\lesssim\delta_{j,k}^{-\frac{T+4}{4}}.\] 
Interpolation between the last two estimates gives
\[\|\scriptA_{j,k}\|_{L^{\frac{8(T-2\di+3)}{5T-4\di+16}}(\R^3)\to L^{\frac{8(T-2\di+3)}{3T-4\di+8}}(\R^3)}\lesssim \delta_{j,k}^{-\frac{(T+4)(\di +1)}{4(T-\di+3)}}.\]
Then, using Lemma \ref{childrenlemma}, we have
\begin{eqnarray*}
\|\tilde{\scriptA}_{j,k}\|_{L^{\frac{8(T-2\di+3)}{5T-4\di+16}}(\R^3)\to L^{\frac{8(T-2\di+3)}{3T-4\di+8}}(\R^3)}
&\lesssim &\delta_{j,k}^{-\frac{(T+4)(\di +1)}{4(T-\di+3)}} (\det D_{j,k})^{-\frac{T+4}{4(T-\di+3)}}\\
&=&\delta_{j,k}^{-\frac{(T+4)(\di +1)}{4(T-\di+3)}} 2^{j(\di+1)(r+1)\frac{T+4}{4(T-\di+3)}}\\
&=& 2^{k\frac{(T+4)(\di +1)}{4(T-\di+3)}} 2^{j\frac{(\di+1)(T+4)}{4(T-\di+3)}}.
\end{eqnarray*}
Clearly $\|\tilde{\scriptA}_{j,k}\|_{L^{1}(\R^3)\to L^{1}(\R^3)}\lesssim1$. In order to have the desired result it is sufficient to observe that 
\[\Sum\limits_{k=0}^{\infty}\Sum\limits_{ 0\leq j \leq\frac{k}{r}} 2^{-k\left[1-\theta\frac{(T+4)(\di +1)}{4(T-\di+3)}\right]} 2^{-j\left[1-\theta\frac{(\di+1)(T+4)}{4(T-\di+3)}\right]} <\infty
\]
for any $0\le\theta<\frac{4(T-\di+3)}{(T+4)(\di+1)}\le 1$. Note that
\[
\frac{4(T-\di+3)}{(T+4)(\di+1)}\le 1 \Longleftrightarrow T(d_h-3) +8(d_h-1) \ge 0,
\]
which is clearly true if $d_h\ge 3$. For $d_h<3$, assume for a contradiction that $T>\frac{8(d_h-1)}{3-d_h}$. Then, since here $w$ cannot be a monomial,
\[
2d_h-2=d_h(w)\ge \frac{rT}{r+1} > \frac{8r(d_h-1)}{(r+1)(3-d_h)} \ge \frac{16(d_h-1)}{3(3-d_h)}.
\]
Thus
\[
1> \frac{8}{3(3-d_h)},
\]
i.e. $1> 3d_h$, which is false.



\section{A relation between the heights of a mixed-homogeneous polynomial and its Hessian determinant}
\begin{lemma}\label{hardcomputation}
Assume that $\varphi(y)=(y_2-\lambda y_1^r)^NQ(y)$, $\lambda\in\R\setminus\{0\}$, $r\ge 2$, $N\ge 2$ and $Q(y_1,\lambda y_1^r)\neq0$ for $y_1\neq0$. Then $w(y)=(y_2-\lambda y_1^r)^{2N-3}\tilde{Q}(y)$, where $\tilde{Q}(y_1,\lambda y_1^r)\neq0$ for $y_1\neq0$.
\end{lemma}

\begin{proof}
Simple computations show that 
\begin{eqnarray*}
\partial_1^2[(y_2-\lambda y_1^r)^N Q(y)] & = & (y_2-\lambda y_1^r)^{N-2}[N(N-1)\lambda^2 r^2 y_1^{2r-2} Q(y) \\
&& - N\lambda r(r-1)(y_2-\lambda y_1^r) y_1^{r-2} Q(y) \\
&& - 2N\lambda r (y_2-\lambda y_1^r) y_1^{r-1} \partial_1 Q(y) \\
&& + (y_2-\lambda y_1^r)^2 \partial_1^2 Q(y)],
\end{eqnarray*}
\begin{eqnarray*}
\partial_1\partial_2 [(y_2-\lambda y_1^r)^N Q(y)] & = & (y_2-\lambda y_1^r)^{N-2} [-N(N-1) \lambda r y_1^{r-1} Q(y) \\
&& +N (y_2-\lambda y_1^r) \partial_1 Q(y) \\
&& -N\lambda r (y_2-\lambda y_1^r) y_1^{r-1} \partial_2 Q(y) \\
&& + (y_2-\lambda y_1^r)^2 \partial_1\partial_2 Q(y)],
\end{eqnarray*}
and
\begin{eqnarray*}
\partial_2^2 [(y_2-\lambda y_1^r)^N Q(y)] & = & (y_2-\lambda y_1^r)^{N-2} [N(N-1) Q(y) \\
&& + 2N(y_2-\lambda y_1^r) \partial_2 Q(y) \\
&& + (y_2-\lambda y_1^r)^2 \partial_2^2 Q(y)].
\end{eqnarray*}
We may clearly factor out $(y_2-\lambda y_1^r)^{2N-4}$ from the determinant defining $w$, therefore it has to be shown that exactly one more factor $(y_2-\lambda y_1^r)$ can be factored out of the remaining determinant. Using the multilinearity of the remaining determinant, we observe that the term which is constant in the factor $(y_2-\lambda y_1^r)$ is
\[
\left|\begin{array}{cc} N(N-1) \lambda^2 r^2 y_1^{2r-2} Q(y) & -N(N-1) \lambda r y_1^{r-1} Q(y) \\
-N(N-1) \lambda r y_1^{r-1} Q(y) & N(N-1) Q(y) \end{array} \right|,
\]
which is equal to $0$. We next consider the terms which contain a single factor in $(y_2-\lambda y_1^r)$. These are
\begin{align*}
& (y_2-\lambda y_1^r) \left|\begin{array}{cc} -N\lambda r(r-1) Q(y) - 2N\lambda r y_1^{r-1} \partial_1 Q(y) & -N(N-1) \lambda r y_1^{r-1} Q(y) \\
N\partial_1 Q(y) - N\lambda r y_1^{r-1} \partial_2 Q(y) & N(N-1) Q(y) \end{array}\right| \\
& + (y_2-\lambda y_1^r) \left|\begin{array}{cc} N(N-1) \lambda^2 r^2 y_1^{2r-2} Q(y) & N\partial_1 Q(y) - N\lambda r y_1^{r-1} \partial_2 Q(y) \\
-N(N-1) \lambda r y_1^{r-1} Q(y) & 2N \partial_2 Q(y) \end{array}\right| \\
= \ \ & N^2(y_2-\lambda y_1^r) [-\lambda r(r-1)(N-1) y_1^{r-2} Q(y)^2 - 2\lambda r(N-1)y_1^{r-1} \partial_1 Q(y) Q(y) \\
& +(N-1)\lambda r y_1^{r-1} \partial_1 Q(y) Q(y) - \lambda^2 r^2 (N-1) y_1^{2r-2} \partial_2 Q(y) Q(y) \\
& + 2(N-1) \lambda^2 r^2 y_1^{2r-2} \partial_2 Q(y) Q(y) + \lambda r (N-1) y_1^{r-1} \partial_1 Q(y) Q(y) \\
& - \lambda^2 r^2 (N-1) y_1^{2r-2} \partial_2 Q(y) Q(y)] \\
= \ \ & - N^2 (N-1) \lambda r(r-1) (y_2-\lambda y_1^r) y_1^{r-2} Q(y)^2.
\end{align*}
Under the assumptions of the lemma, the last expression is a nonzero multiple of $(y_2-\lambda y_1^r)$. Note that in the homogeneous case $r=1$, the expression is equal to zero and one may factor out an additional power of $(y_2-\lambda y_1^r)$ as in \cite{FGU06}. The rest of the terms are of higher order in $(y_2-\lambda y_1^r)$ and so the proof of Lemma \ref{hardcomputation} is complete.
\end{proof}

The above lemma, together with Lemma \ref{computationsforw} and the discussion in Section \ref{s>1}, allows us to deduce a relation between the height of $w$ in terms of the height of $\varphi$ or in terms of the Taylor support of $\varphi$ in certain cases. For $\varphi$ a mixed homogenous (and in particular not a homogeneous) polynomial, if $N \ge d_h(\varphi) + 1/2$, then $h(w)=2N-3=2\h-3$. If $N < d_h(\varphi) +1/2$ and $\max\{\nu_1,\nu_2\}\ge d_h(\varphi)$, then $\h=\nu_i$ for $i\in\{1,2\}$ and $h(w)=2\nu_i -2=2\h -2$. If $N < d_h(\varphi) +1/2$, $s\ge 2$ and $\max\{\nu_1,\nu_2\}< d_h(\varphi)$, then $h(w)=2d_h(\varphi) -2 = 2\h -2$, unless the conditions of \ref{A>2d} hold in which case $h(w)=A-2$.

\end{document}